\documentclass[11pt]{article}
\usepackage{epigamath}

\setpapertype{A4}

\usepackage[english]{babel}

\usepackage[dvips]{graphicx}     

\title{\vspace{-1.5cm}Haas' Theorem revisited}
\titlemark{Haas' Theorem revisited}
\author{\vspace{0.5cm}Beno\^{\i}t Bertrand, Erwan Brugall\'e, and Arthur Renaudineau}
\authoraddresses{
\authordata{Beno\^{\i}t Bertrand}{\firstname{Beno\^{\i}t} \lastname{Bertrand}\\
\institution{Institut de math\'ematiques de Toulouse -- IUT de Tarbes, 1 rue Lautr\'eamont -- CS 41624, 65016 Tarbes, France 
}\\
\email{benoit.bertrand@math.univ-toulouse.fr}}

\authordata{Erwan Brugall\'e}{\firstname{Erwan} \lastname{Brugall\'e}\\
\institution{\'Ecole polytechnique, CNRS, Universit\'e Paris-Saclay, 91128 Palaiseau Cedex, France}\\
\email{erwan.brugalle@math.cnrs.fr}}

\authordata{Arthur Renaudineau}{\firstname{Arthur} \lastname{Renaudineau}\\
\institution{Eberhard Karls Universit\"at T\"ubingen, Fachbereich Mathematik, Institut f\"ur Geometrie, Germany}\\
\email{arren@math.uni-tuebingen.de}}
}
\authormark{B. Bertrand, E. Brugall\'e \& A. Renaudineau}
\date{\vspace{-5ex}} 
\journal{\'Epijournal de G\'eom\'etrie Alg\'ebrique} 
\acceptation{Received by the Editors on September 14, 2016, and in final form on July 3, 2017. \\ Accepted on July 28, 2017.}

\newcommand{\articlereference}{Volume 1 (2017), Article Nr. 9}

\newtheorem{theorem}{Theorem}[section]
\newtheorem{lemma}[theorem]{Lemma}
\newtheorem{remark}[theorem]{Remark}
\newtheorem{corollary}[theorem]{Corollary}
\newtheorem{example}[theorem]{Example}
\newtheorem{definition}[theorem]{Definition}
\newtheorem{proposition}[theorem]{Proposition}

\newcommand{\R}{\mathbb{R}}
\newcommand{\RR}{\mathbb{R}}
\newcommand{\RP}{\mathbb{RP}}
\newcommand{\Z}{\mathbb{Z}}

\newcommand{\C}{\mathbb{C}}
\newcommand{\A}{\mathcal{A}}
\newcommand{\CC}{\mathbb{C}}

\renewcommand{\SS}{\mathcal{S}}
\renewcommand{\epsilon}{\varepsilon}
\renewcommand{\P}{\mathcal{P}}
\newcommand{\Ve}{\text{Vert}}
\newcommand{\Ed}{\text{Edge}}
\newcommand{\val}{\text{val}}
\newcommand{\td}{{\text{''}}}
\newcommand{\tg}{{\text{``}}}
\newcommand{\Log}{{\text{Log}}}
\newcommand{\Arg}{{\text{Arg}}}
\newcommand{\conj}{\mathop{\mathrm{conj}}}
\newcommand{\Id}{\mathop{\mathrm{Id}}}
\newcommand\restrict[1]{\raisebox{-.9ex}{$\big|$}_{ ^{#1}}}

\newcommand\restr[2]{{
  \left.\kern-\nulldelimiterspace 
  #1
  \vphantom{\big|} 
  \right|_{#2} 
}}

\newcommand{\RNum}[1]{\uppercase\expandafter{\romannumeral #1\relax}}

\begin{document}


\maketitle



\begin{prelims}

\vspace{-0.55cm}

\def\abstractname{Abstract}
\abstract{Haas' theorem describes all patchworkings of a given non-singular plane tropical curve $C$ giving rise to a maximal real algebraic curve.
The space of such patchworkings
is naturally 
a linear subspace $W_C$ of the $\Z/2\Z$-vector space
$\overrightarrow \Pi_C$
generated by the bounded edges of $C$, and whose origin is the
Harnack patchworking. 
The aim of this note is to provide an interpretation of affine subspaces  of
$\overrightarrow \Pi_C $ parallel to $W_C$.
\\
To this purpose, we
 work in the setting of abstract graphs rather than plane tropical curves.
We introduce a 
topological
surface $S_\Gamma$ above a trivalent graph
  $\Gamma$,
and consider a suitable affine space $\Pi_\Gamma$ of
 real structures on $S_\Gamma$ compatible with $\Gamma$.
We 
identify the vector subspace 
$W_\Gamma$
of 
$\overrightarrow \Pi_\Gamma$ 
characterizing 
real structures inducing
the same
 action on 
$H_1(S_\Gamma,\Z/2\Z)$.
We then deduce from this statement another proof of Haas' original result.}



\vspace{-0.05cm}

\languagesection{Fran\c{c}ais}{%

\vspace{-0.05cm}
\textbf{Titre. Th\'eor\`eme de Haas revisit\'e}
\commentskip
\textbf{R\'esum\'e.}
Le th\'eor\`eme de Haas d\'ecrit tous les patchworks d'une courbe tropicale plane lisse donn\'ee $C$ donnant lieu \`a une courbe alg\'ebrique r\'eelle maximale. L'espace de ces patchworks est naturellement un sous-espace lin\'eaire $W_C$ du $\Z/2\Z$-espace vectoriel
$\overrightarrow \Pi_C$
engendr\'e par les ar\^etes born\'ees de $C$, et dont l'origine est le patchwork de Harnack. 
Le but de cette note est de donner une interpr\'etation des sous-espaces affines de 
$\overrightarrow \Pi_C $ parall\`eles~\`a~$W_C$.
\\
Pour ce faire, nous nous pla\c{c}ons dans le cadre des graphes
abstraits plut\^ot que celui des courbes tropicales planes. Nous introduisons une surface topologique $S_\Gamma$ au dessus d'un graphe trivalent $\Gamma$,
et consid\'erons un espace affine ad\'equat $\Pi_\Gamma$ de structures r\'eelles sur $S_\Gamma$ compatibles avec $\Gamma$.
Nous identifions le sous-espace vectoriel
$W_\Gamma$
de
$\overrightarrow \Pi_\Gamma$ 
caract\'erisant des structures r\'eelles induisant la m\^eme action sur
$H_1(S_\Gamma,\Z/2\Z)$.
Nous d\'eduisons alors de cet \'enonc\'e une autre d\'emonstration du r\'esultat original de Haas.}

\end{prelims}


\newpage

\setcounter{tocdepth}{1}
\tableofcontents

\section{Introduction}

Haas classified  in his thesis ~\cite{Haa2}
all unimodular combinatorial patchworkings
(see \cite{Viro,IV2})
producing a real algebraic
$M$-curve (i.e. a non-singular compact real algebraic curve of
genus $g$ whose
real part has $g+1$ connected components). 
In the tropical reformulation of
patchworking  in
terms of \emph{twist-admissible} 
sets of edges of a non-singular plane
tropical curve $C$ given in \cite{BIMS15},
the set of all possible patchworkings with this  given underlying
 tropical curve $C$ is naturally  a vector
 space  $\overrightarrow \Pi_C$ over $\Z/2\Z$
 and
Haas' Theorem can be interpreted
as follows:
 the space of all such patchworkings  producing an $M$-curve is
 a
well identified and easily described
 subvector space  $W_C$ of  
 $\overrightarrow \Pi_C$
 (see Section \ref{sec:viro}).
\medskip
In this note we address the question of 
 interpreting  affine subspaces
 of $\overrightarrow \Pi_C $ parallel to $W_C$.
To this purpose, it is suitable
to follow
 Klein's approach
and to work in the framework of 
 abstract real topological surfaces, i.e. oriented topological surfaces 
equipped with an orientation-reversing continuous involution, 
 rather than restricting to  real algebraic curves in a given toric
 surface (see~\cite{Klein-harnack-bound}).
Accordingly, we deal with abstract graphs rather than with plane
tropical curves.
Given an abstract graph $\Gamma$ with only 3-valent and 1-valent
vertices, we
construct a topological surface $S_\Gamma$ decomposed into a union of
disks, cylinders, and pairs of pants 
(this is just a 
variation on standard pair of pants decompositions of a
surface). Next, given a continuous involution $\tau:\Gamma\to\Gamma$, we  define
\emph{real structures} 
 above the pair
$(\Gamma,\tau)$, which are roughly real  
structures on $S_\Gamma$ compatible with the decomposition induced by
$\Gamma$ together with $\tau$ (see Section \ref{sec:real graphs} for precise definitions).
The set of real structures 
above
$(\Gamma,\tau)$ is naturally
an affine space $\Pi_{(\Gamma,\tau)}$
over $\Z/2\Z$ whose direction  $\overrightarrow{\Pi_{(\Gamma,\tau)}}$
has for basis the set of edges of the quotient graph
$\Gamma/\tau$ 
that are adjacent to two 3-valent vertices (see Lemma
\ref{lem:patch aff}).
As any involution, a real structure on  $S_\Gamma$ induces an
action on $H_1(S_\Gamma;\Z/2\Z)$
 and  the main result of this note
can be summarized as follows.
\begin{center}
 \emph{Two real structures
   above $(\Gamma,\tau)$ induce the same action on 
   $H_1(S_\Gamma;\Z/2\Z)$ if and only if they differ by an element of a 
   given 
   vector subspace $W_{(\Gamma,\tau)}$ of
   $\overrightarrow{\Pi_{(\Gamma,\tau)}}$  having a simple description}. 
\end{center}
This statement is proved in Theorem~\ref{thm:main} after the definition of $W_{(\Gamma,\tau)}$ at the begining of Section~\ref{sec:action}.
 We show in
Proposition~\ref{prop:equiv W} that $W_{(\Gamma,\tau)}$  admits
an alternative description which corresponds to Haas' description of
the above-mentioned vector space $W_C$ in the special case of 
a non-singular plane tropical curve $C$.
The connection of  Theorem~\ref{thm:main} and Proposition~\ref{prop:equiv W} 
with Haas' Theorem then comes  from  the fact that
 a real
 topological surface with a non-empty real part
 is maximal if and only if the corresponding
 induced action on the first homology group
 of the surface is trivial.
In particular, we recover Haas' Theorem as a corollary of our main
results combined with
standard results in tropical geometry (see Section~\ref{sec:haas}).

\medskip
Our main motivation for the present work was the possible generalisations of Haas' Theorem in higher dimensions.
Generalising a statement
usually first requires
to identify the suitable notions coming into play and allowing
a suitable
formulation of the original statement.
Both in its original formulation and in its  ``twist-admissible''
tropical reformulation, Haas' Theorem 
involves, sometimes implicitely,
several features that a priori make sense only for
curves.
As an example, it is based on the
existence of Harnack distribution of signs discovered by Itenberg (see
\cite{IV2}), which has no known analogue
yet in higher dimensions.
The existence of such Harnack distribution of signs on $\Z^2$ 
is precisely the
 fact
 that naturally turns
the space of all patchworkings of a non-singular
tropical curve in $\R^2$ in a vector  space rather than an
affine space 
as  in the case of abstract graphs: it provides a canonical
patchworking on any non-singular
tropical curve in $\R^2$, which 
in addition
produces an $M$-curve.
By focusing on abstract graphs rather than on embedded tropical curves,
and on 
the action on homology induced by a real structure
rather than the number of connected components of its real part, we
place Haas' Theorem in a wider perspective. 
There, it turns out to be a corollary of a more general statement which seems to us more likely
to have a higher dimensional analogue.

\subsubsection*{A few words on the context}
Despite the surprising elegance of the complete description of maximal unimodular patchworkings, Haas' Theorem unfortunately only ever appeared in his thesis~\cite{Haa2}. We briefly recap its
 origin, rooted in the first part of Hilbert's
16th problem. This latter concerns the classification, up to
isotopy, 
 of all possible mutual positions of the connected  components of 
 the real part of a
 non-singular real algebraic curve in $\RP^2$ of a given degree $d$. If $d$ is even, every connected component of the real part of such a curve
 is called an \emph{oval}, and bounds a
 disc which is called the \emph{interior} of the oval.  One says that an oval is \emph{even} (resp. \emph{odd}) if it 
 is contained in the interior of an even (resp. odd) number of ovals. The following Ragsdale conjecture played an important role in  subsequent developments in real algebraic geometry, and remains one 
challenging open
question in the case of $M$-curves.

\bigskip
\noindent {\bf Conjecture} (Refined
    Ragsdale conjecture,  \cite{Rag,Pet}){\bf .} {\em For any  non-singular
  real algebraic curve of even degree $2k$ in $\RP^2$
  having $p$ even  and $n$ odd ovals, one has 
$$
p\leq \dfrac{3k(k-1)}{2}+1\qquad  \mbox{and}\qquad  n\leq \dfrac{3k(k-1)}{2}+1.
$$}

\medskip
A series of counterexamples to this conjecture have been constructed
since the 90's \cite{I3,Haa,Br2}. Nevertheless, no counterexample is
known yet among
$M$-curves.
As an application of his classification of maximal unimodular combinatorial
patchworkings, Haas proved in his thesis the following theorem.

\bigskip
\noindent
{\bf Theorem} (Haas, {\cite[Theorem 12.4.0.12 and Proposition 13.5.0.13]{Haa2}}){\bf .}
{\em Let $A$ be a non-singular real algebraic $M$-curve in $\RP^2$
constructed by a unimodular combinatorial patchworking. 
 If $A$ has $p$ even and $n$ odd ovals, then one has 
$$
p\leq \dfrac{3k(k-1)}{2}+1 \qquad \mbox{and}\qquad  n\leq \dfrac{3k(k-1)}{2}+4.
$$
Furthermore, such
a curve having more than $\dfrac{3k(k-1)}{2}+1$ odd ovals would have exactly $n= \dfrac{3k(k-1)}{2}+4$ such ovals.}

\bigskip

 As far as we know, 
 this is the only known result
 in the direction of Ragsdale conjecture for maximal curves.

\subsubsection*{Organisation of the paper}We start by defining real structures above an abstract graph in Section~\ref{sec:real patch}. In
 Section~\ref{sec:action}, we  prove our main statement
 (Theorem~\ref{thm:main})
  and Proposition~\ref{prop:equiv W} which
 relates our definition of $W_{(\Gamma,\tau)}$ to Haas' description
 of the vector space $W_C$.
 We end this paper by recalling in Section \ref{sec:patch&haas}
 the combinatorial patchworking construction and Haas' Theorem, and by 
 deducing this latter from results from Section~\ref{sec:action}.

\subsubsection*{Acknowledgements}Part of this work has been done during 
visits of B.B. and A.R. at 
Centre Math\'ematiques Laurent Schwartz, \'Ecole Polytechnique. We thank
this institution for excellent working conditions.
 We also thank Ilia Itenberg for fruitful discussions,
as well as 
  anonymous referees for many valuable
comments on an earlier version of this paper.
 
\subsubsection*{Convention and notation}
 A real topological surface $(S,\tau)$ is 
 an oriented topological surface $S$ 
 equipped with
 an orientation-reversing continuous involution
 $\tau:S\to S$,
 called a \emph{real structure}.  
 Given a  real surface $(S,\tau)$, we denote by $\R S$ the \emph{real part} of $S$, i.e. the set
 of fixed points of $\tau$. It is a submanifold of $S$ of dimension 1.
 If $S$ is compact and of genus $g$,
then $\R S$ has at most $g+1$ connected components (see~\cite{Klein-harnack-bound}) and we say that $(S,\tau)$ is \emph{maximal} if it has $g+1$ connected components. 
If $(S,\tau)$ is a compact real surface and $A\subset S$ is a real finite
set of points, we say that  $(S\setminus A,\tau)$ is maximal if
$(S,\tau)$ is.

\smallskip
Given a finite graph $\Gamma$, 
 we denote 
by $\Ve(\Gamma)$ the set of its vertices, by $\Ve^\infty(\Gamma)$ the set of its
1-valent vertices, and we set  $\Ve^0(\Gamma)=\Ve(\Gamma)\setminus
\Ve^\infty(\Gamma)$. By definition, the valency of a vertex
 $v\in\Ve(\Gamma)$, denoted by $\val(v)$,
 is the number of edges of $\Gamma$ adjacent to $v$. We also denote by  
$\Ed(\Gamma)$ the set of edges of $\Gamma$, and by $\Ed^0(\Gamma)$ the
 set of edges 
of $\Gamma$  adjacent to two vertices  in $\Ve^0(\Gamma)$. 
Throughout the text
we   identify a graph 
 with any of its topological realisations in which edges are open segments.

\section{Real structures above real graphs}\label{sec:real patch}
\subsection{Surfaces associated to a trivalent graph}\label{sec:graph surf}

We will call \emph{trivalent} a graph such that any of its  vertices
is either $3$-valent or $1$-valent,
thus authorizing that it also has leaves. 
 Given a  trivalent graph $\Gamma$, we construct a topological surface 
 $S_\Gamma$ as follows.
 Recall that a pair of pants is an oriented sphere with three open disks removed, in particular it has three boundary components with induced orientation.

\begin{enumerate}
\item To each vertex $v$ of $\Gamma$, we associate a topological surface $S_v$,
which is either a pair of pants if $v$ is 3-valent, or an oriented
closed disk if 
$v$ is 1-valent. Furthermore, we choose 
 a one to one correspondence between boundary components of $S_v$ and
 edges of $\Gamma$ 
adjacent to $v$; the boundary component corresponding to $e$ is
denoted by $\gamma_{v,e}$. 

\item To each edge $e$ of $\Gamma$, we associate an oriented cylinder $S_e$, and a
  one to one correspondence between boundary components of $S_e$ and
  vertices adjacent to $e$; the boundary component corresponding to $v$ is
denoted by $\gamma_{e,v}$. 

\item To each pair $(v,e)\in \Ve(\Gamma)\times \Ed(\Gamma)$ such that $e$ is
  adjacent to $v$, we associate 
 an orientation-reversing homeomorphism $\phi_{v,e}: \gamma_{v,e}\to
\gamma_{e,v}$. 

\end{enumerate}

The surface $S_\Gamma$ is obtained by gluing all surfaces $S_v$ and
$S_e$ via the maps $\phi_{v,e}$. It is
a closed oriented topological surface.
Clearly, the surface $S_\Gamma$ is not uniquely defined by $\Gamma$, but also depends
on  the choice of the homeomorphisms $\phi_{v,e}$
and of the surfaces $S_v$. However, different
 surfaces obtained by different choices are 
 homeomorphic, and such an homeomorphism is canonical up to composition
 by Dehn twists 
along the circles $\gamma_{e,v}$ in $S_\Gamma$.

\begin{example}\label{ex:genus 2}
{\rm  We depicted in Figure \ref{fig:genus 2} a
  trivalent graph $\Gamma$ and the
corresponding surface $S_\Gamma$.}
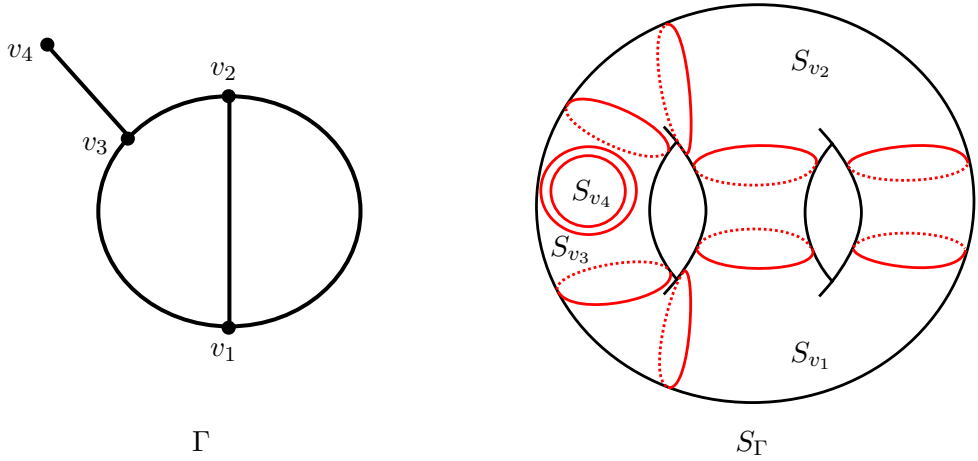
\begin{figure}[h!]
\begin{center}
\begin{tabular}{ccc}
\includegraphics[width=5cm, angle=90]{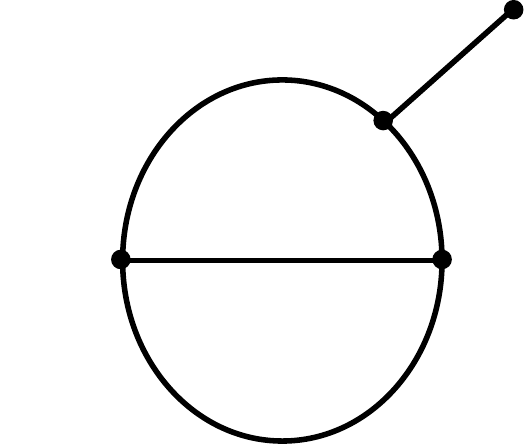}
\put(-57,22){$v_1$}
\put(-57,128){$v_2$}
\put(-105,100){$v_3$}
\put(-133,135){$v_4$}
&\hspace{8ex} 
&\includegraphics[width=5.5cm, angle=90]{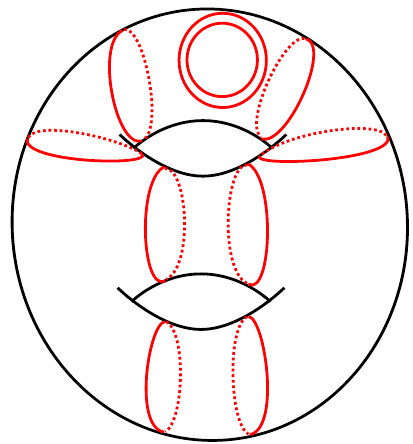}
\put(-70,20){$S_{v_1}$}
\put(-70,130){$S_{v_2}$}
\put(-160,60){$S_{v_3}$}
\put(-152,81){$S_{v_4}$}
\\   $\Gamma$
&&  $S_\Gamma$
\end{tabular}
\end{center}
\caption{A trivalent graph $\Gamma$ and the corresponding surface $S_\Gamma$}
\label{fig:genus 2}
\end{figure}
\end{example}

\subsection{The groups $H_{1,0}(S_\Gamma;\Z/2\Z)$ and $H_{0,1}(S_\Gamma;\Z/2\Z)$}
Let $\Gamma$ be a 
 trivalent graph.
Each
circle $\gamma_{e,v}$ defines an element in $H_1(S_\Gamma;\Z/2\Z)$.
If $v$ and $v'$ are the two vertices adjacent to $e$, both circles $\gamma_{e,v}$ and $\gamma_{e,v'}$ define the same element in $H_1(S_\Gamma;\Z/2\Z)$, that we denote by $\gamma_{e}$. 
We define $H_{1,0}(S_\Gamma;\Z/2\Z)$ to be the subgroup of $H_1(S_\Gamma;\Z/2\Z)$
generated by all the loops $\gamma_{e}$, and we set
$$H_{0,1}(S_\Gamma;\Z/2\Z)=H_{1}(S_\Gamma;\Z/2\Z)/H_{1,0}(S_\Gamma;\Z/2\Z).$$

\noindent Let $\alpha_1,\ldots,\alpha_g$ be a basis of $H_1(\Gamma;\Z/2\Z)$.
To each 
1-cycle  $\alpha_i$, 
 we associate a class $\gamma_{\alpha_i}$ in $H_1(S_\Gamma;\Z/2\Z)$ as follows:

\begin{enumerate}
\item for each $e\in\Ed(\Gamma)$ such that $e\subset\alpha_i$, choose a
  $1$-chain $\eta_e$ in $S_e$ joining the two boundary components of $S_e$. Denote by $\eta_{v,e}$
the boundary point of $\eta_e$ contained in $\gamma_{e,v}$;

\item for each $v\in\Ve(\Gamma)$ such that $v\in\alpha_i$, choose a $1$-chain $\eta_v$  in $S_v$ 
 joining the points $\eta_{v,e}$ and $\eta_{v,e'}$, 
 where $e$ and $e'$ are the two edges contained in $\alpha_i$ and adjacent to $v$.
\end{enumerate}

We denote by $\gamma_{\alpha_i}\in H_1(S_\Gamma;\Z/2\Z)$ the class defined by 
the union of all $1$-chains $\eta_e$ and $\eta_v$ above.
Further, we extend 
the map $\alpha_i\mapsto\gamma_{\alpha_i}$ by linearity
to 
an injective $\Z/2\Z$-linear
map 
$$\begin{array}{ccc}
H_{1}(\Gamma;\Z/2\Z) & \longrightarrow &
H_{1}(S_\Gamma;\Z/2\Z)
\\ \alpha &\longmapsto & \gamma_\alpha
\end{array}.$$
Note that $\gamma_\alpha$ is not uniquely defined by $\alpha$,
but is well defined only up to an element in $H_{1,0}(S_\Gamma;\Z/2\Z)$. In
other words, there is a natural and well defined 
 injective
$\Z/2\Z$-linear map (or group morphism)
 $$H_{1}(\Gamma;\Z/2\Z) 
 \longrightarrow H_{0,1}(S_\Gamma;\Z/2\Z)$$
 that associates to $\alpha$ the class realised by $\gamma_\alpha$ in
$H_{0,1}(S_\Gamma;\Z/2\Z)$. 
The next lemma shows in particular that this map is an isomorphism.

\begin{example}
{\rm We consider the trivalent graph from Example \ref{ex:genus 2}.
We depicted in Figure \ref{fig:genus 2 bis} the 1-cycle $\gamma_e$,
and two possible 1-cycles $\gamma_\alpha$, denoted by $\gamma_\alpha$
and $\gamma'_\alpha$ (recall that the
decomposition of $S_\Gamma$ into the union of the surfaces $S_v$ and $S_e$
is depicted in Figure \ref{fig:genus 2}).}
\begin{figure}[h!]
\begin{center}
\begin{tabular}{ccc}
\includegraphics[width=5cm, angle=90]{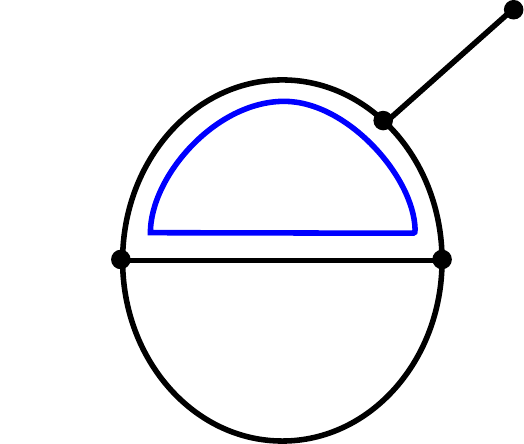}
\put(5,75){$e$}
\put(-80,75){$\alpha$}
&\hspace{8ex} 
&\includegraphics[width=5.5cm, angle=90]{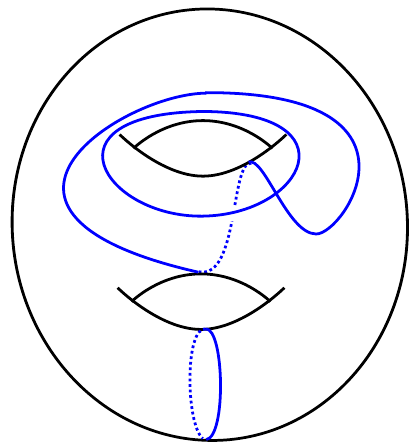}
\put(5,75){$\gamma_e$}
\put(-105,55){$\gamma_\alpha$}
\put(-80,25){$\gamma'_\alpha$}
\\   $\Gamma$
&&  $S_\Gamma$
\end{tabular}
\end{center}
\caption{Lifting  cycles}
\label{fig:genus 2 bis}
\end{figure}
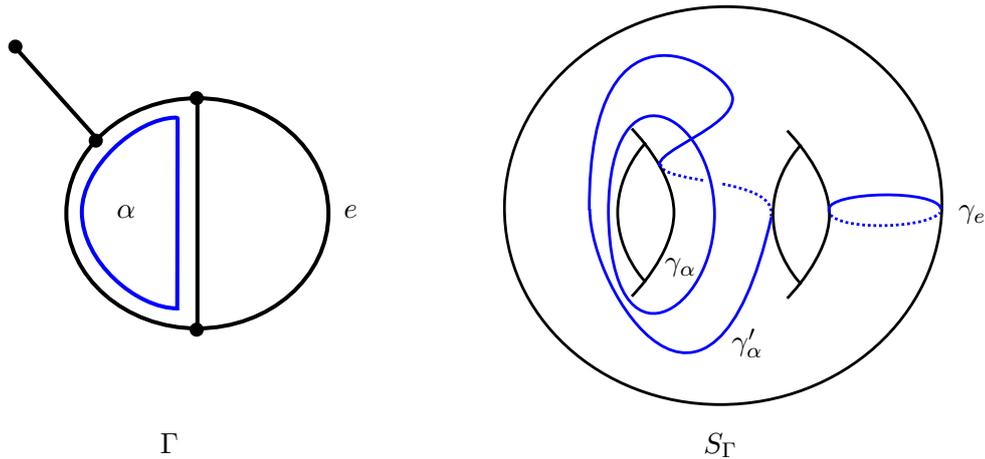
\end{example}

\begin{lemma}\label{lem:hodge dec}
Let $\Gamma$ be a trivalent graph with $b_1(\Gamma)=g$.
Then both $\Z/2\Z$-vector spaces  $H_{1,0}(S_\Gamma;\Z/2\Z)$ and
$H_{0,1}(S_\Gamma;\Z/2\Z)$ have dimension  $g$.
Furthermore, the intersection form on $H_1(S_\Gamma;\Z/2\Z)$ vanishes on
$H_{1,0}(S_\Gamma;\Z/2\Z)$. 
\end{lemma}
\begin{proof}
The group $H_{1,0}(S_\Gamma;\Z/2\Z)$ is
clearly contained in its orthogonal for the intersection form on $H_{1}(S_\Gamma;\Z/2\Z)$.
Since this latter  
 is non-degenerate, we deduce that $H_{1,0}(S_\Gamma;\Z/2\Z)$ has dimension at most $g$.
 Let  $U$ be the vector subspace of $H_{1}(S_\Gamma;\Z/2\Z)$ of
 dimension
 $g$ generated by
all classes $\gamma_{\alpha}$ with $\alpha\in H_{1}(\Gamma;\Z/2\Z)$. 
 By
construction, the intersection of $U$ with the orthogonal of
$H_{1,0}(S_\Gamma;\Z/2\Z)$ is trivial, from which 
we deduce that $H_{1,0}(S_\Gamma;\Z/2\Z)$ has dimension at least $g$. Hence
$H_{1,0}(S_\Gamma;\Z/2\Z)$ has dimension  $g$.
Since by definition, the group $H_{0,1}(S_\Gamma;\Z/2\Z)$ is the quotient
of $H_{1}(S_\Gamma;\Z/2\Z)$ by $H_{1,0}(S_\Gamma;\Z/2\Z)$, it also has  dimension
$g$. 
\hfill $\Box$
\end{proof}

We will abusively identify 
 $H_{0,1}(S_\Gamma;\Z/2\Z)$  with the subgroup of $H_1(S_\Gamma;\Z/2\Z)$ generated by 
 all classes $\gamma_{\alpha}$ with $\alpha\in H_{1}(\Gamma;\Z/2\Z)$. We then have  the decomposition
$$H_{1}(S_\Gamma;\Z/2\Z) =   H_{1,0}(S_\Gamma;\Z/2\Z)\oplus
H_{0,1}(S_\Gamma;\Z/2\Z). $$ 
 Hence the group $H_{0,1}(S_\Gamma;\Z/2\Z)$ viewed as a subgroup of
$H_1(S_\Gamma;\Z/2\Z)$ 
is the image of a section of the
quotient map 
$H_{1}(S_\Gamma;\Z/2\Z) \to H_{0,1}(S_\Gamma;\Z/2\Z)$.

\begin{remark}
{\rm Lemma \ref{lem:hodge dec} and the filtration of  $H_{1}(S_\Gamma;\Z/2\Z)$
by $H_{1,0}(S_\Gamma;\Z/2\Z)$ and  $H_{0,1}(S_\Gamma;\Z/2\Z)$ may be
seen as very particular
cases of  \cite[Theorem 1]{IKMZ}. This relation to tropical homology
explains our  choice of notation.}
\end{remark}

\subsection{Real trivalent graphs and real structures above
  them}\label{sec:real graphs}
\begin{definition}
{\rm A \emph{real trivalent graph} is a pair $(\Gamma,\tau)$, where $\Gamma$ is
a trivalent graph and $\tau:\Gamma\to \Gamma$ a continuous
 involution (called real structure), such that the
restriction of $\tau$ on any (open) edge of $\Gamma$ is either the identity or has no fixed points.}
\end{definition}

The condition on the restriction of $\tau$ on edges of $\Gamma$ is to
exclude the situation when this restriction is a symmetry locally given by
$x\mapsto -x$. 

\begin{example}
{\rm Any trivalent graph $\Gamma$ carries a canonical real structure given 
by $\tau=\Id$.}
\end{example}

\begin{example}\label{ex:twist}
{\rm Two real trivalent graphs $(\Gamma,\tau)$ are depicted in
Figure \ref{fig:twist}. The graph on the left is of genus~$2$ whereas the other one is of genus~$3$. In both cases the graph $\Gamma$ is drawn on the plane, and $\tau$ is
the axial symmetry with respect to the line supporting edges $e_1$ (resp. $e_1$ and
$e_2$). Hence $e_1$ (resp. $e_1$ and $e_2$) are exactly the $\tau$-invariant edges
of
$\Gamma$.}
\begin{figure}[h!]
\begin{center}
\begin{tabular}{ccc}
  \includegraphics[width=4.5cm, angle=0]{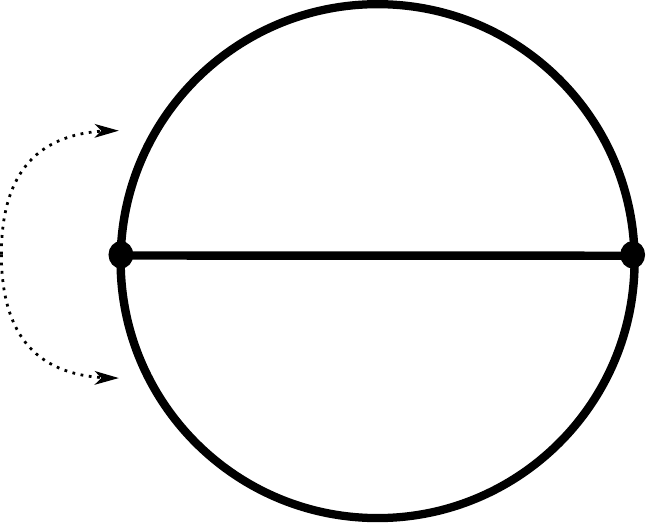}
  \put(-140,50){$\tau$}
  \put(-50,42){$e_1$}
&\hspace{4ex} 
&\includegraphics[width=4.5cm, angle=0]{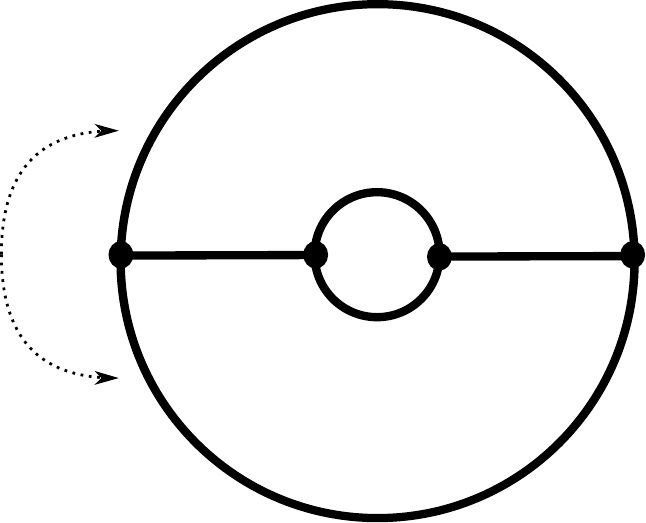}
\put(-140,50){$\tau$}
\put(-87,42){$e_1$}
\put(-30,42){$e_2$}
\\
\\ a)  A real trivalent graph $(\Gamma,\tau)$ of genus $2$
&&  b) A real trivalent graph $(\Gamma,\tau)$ of genus $3$
\end{tabular}
\end{center}
\caption{Real graphs} 
\label{fig:twist}
\end{figure}
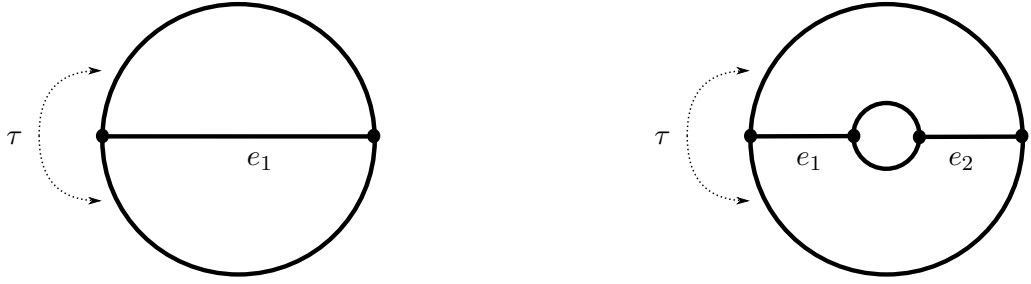
\end{example}

Next we define the lifts of $\tau$ to $S_\Gamma$ that we consider in this
text.
For this purpose, we need to fix the following:

\begin{itemize}
\item  the closed disk $\SS_1=\{z\in\C, \ |z|\le 1\}$ equipped with the
  complex conjugation $\sigma_1:\SS_1\to \SS_1$; 

\item a pair of pants $\SS_3$;

\item two orientation-reversing involutions $\sigma_3:\SS_3\to \SS_3$ and 
$\widetilde \sigma:\SS_3\to \SS_3$ such that the fixed locus of $\sigma_3$
  is three disjoint segments, and the fixed locus of $\widetilde \sigma$
  is a segment (see Figure \ref{fig:real pairs of pants});
\begin{figure}[h!]
\begin{center}
\begin{tabular}{ccc}
\includegraphics[width=4.5cm]{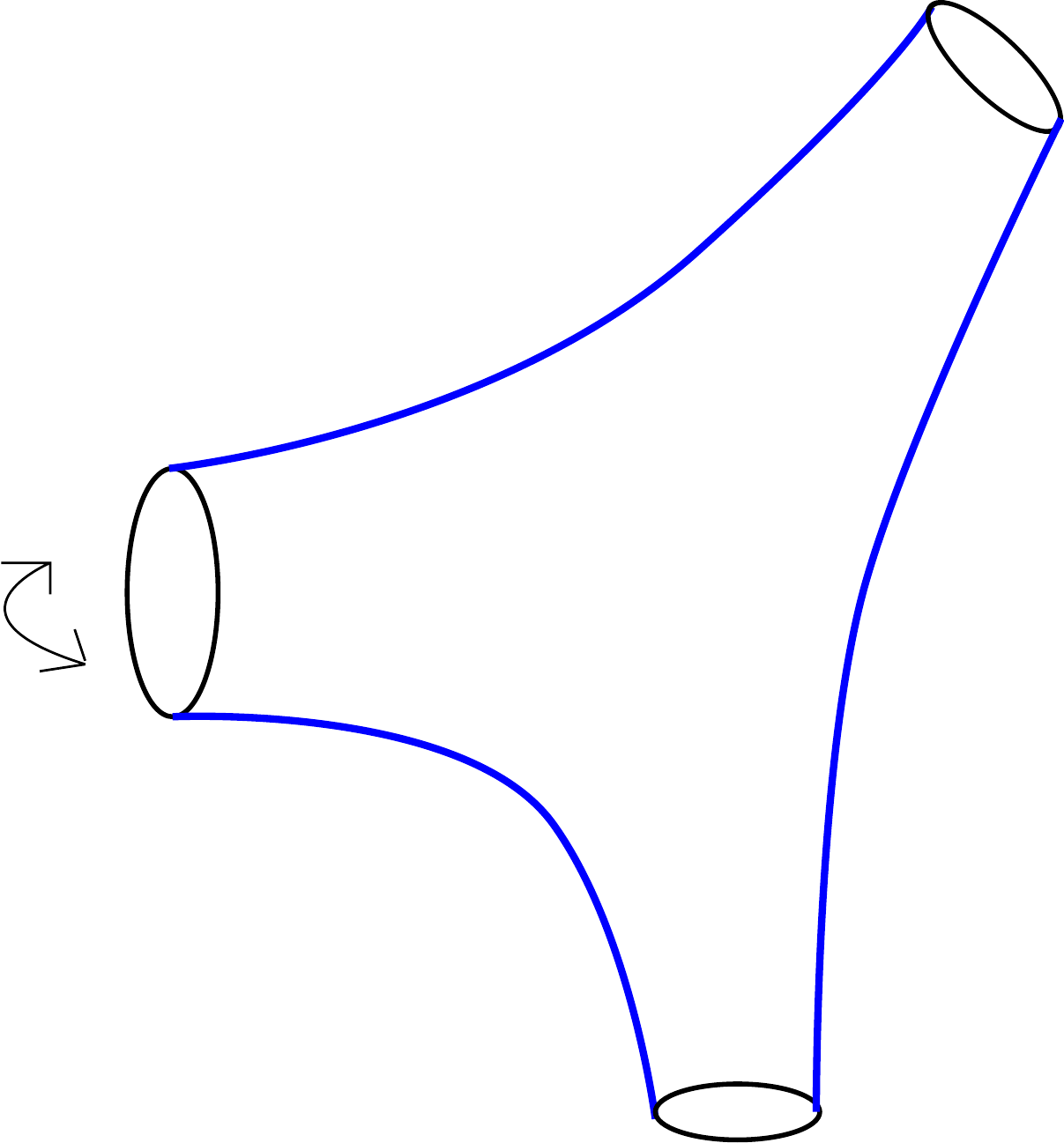}
\put(-145,60){$\sigma_3$}
&\hspace{20ex} 
&\includegraphics[width=4.5cm]{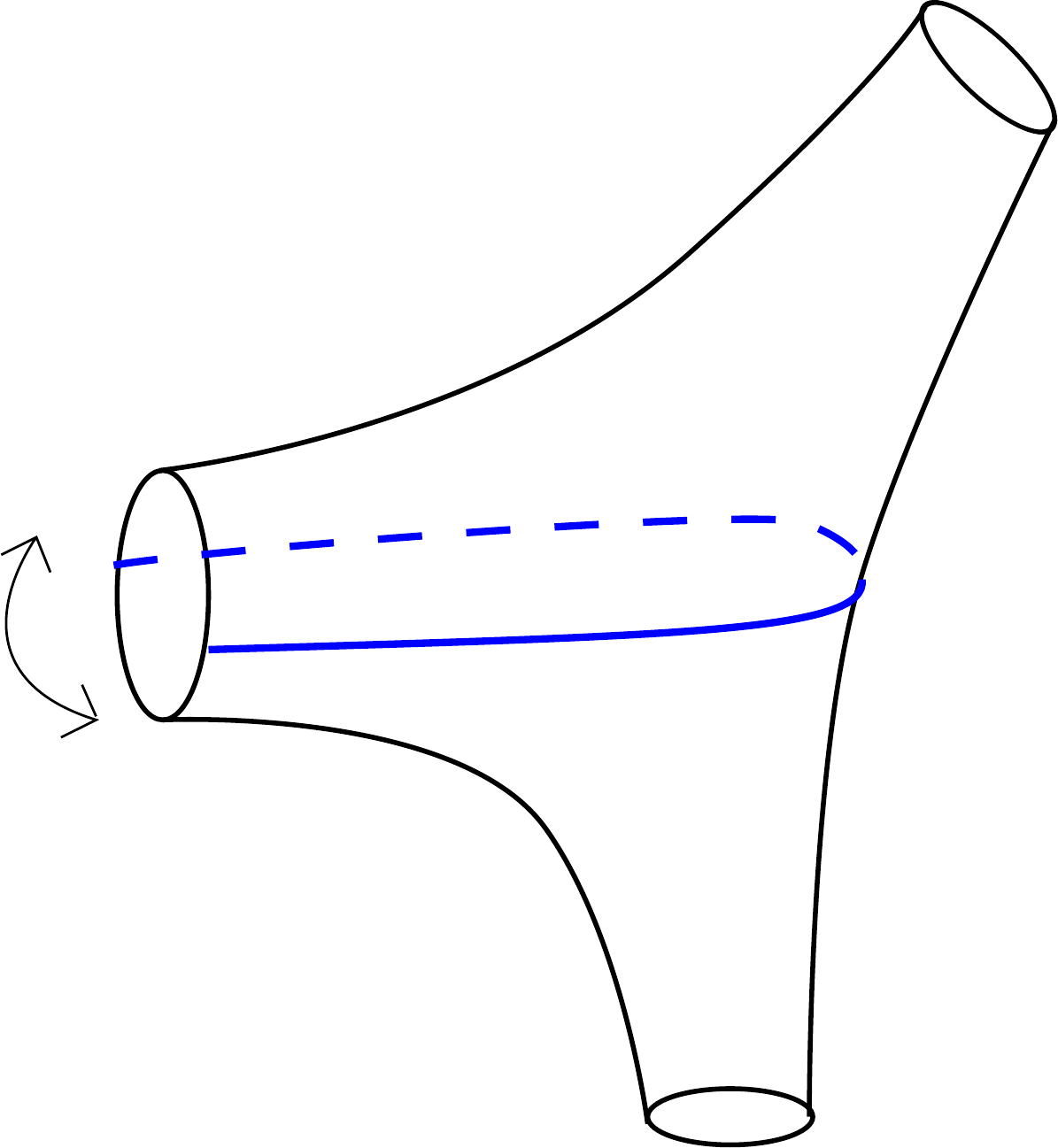}
\put(-145,60){$\widetilde{\sigma}$}
\end{tabular}
\end{center}
\caption{Two real structures on a pair of pants}
\label{fig:real pairs of pants}
\end{figure}
\item for each vertex $v$ of $\Gamma$, an orientation-preserving
  homeomorphism $\psi_v:S_v\to \SS_{\val(v)}$;
  if in addition
  $\tau$ interchanges two adjacent edges of a $\tau$-invariant vertex
   $v$, then
  $\psi_v$ is chosen so that 
$\widetilde{\sigma}(\psi_v(\gamma_{v,e}))=\psi_v(\gamma_{v,\tau(e)})$
for any edge $e$ adjacent to $v$.
\end{itemize}

We say that a continuous orientation-reversing involution $\tau_\Gamma:S_\Gamma\to S_\Gamma$
\emph{lifts} $\tau:\Gamma\to \Gamma$ if the following conditions are satisfied: 
\begin{enumerate}
\item if $\tau$ is locally the identity around $v$, then
   ${\tau_{\Gamma}}\restrict{S_v}= \psi_v^{-1}\circ \sigma_{\val(v)} \circ \psi_v$;   

\item if $v$ is a 3-valent vertex fixed by $\tau$, and if $\tau$
  interchanges two adjacent edges of $v$, then
   ${\tau_{\Gamma}}\restrict{S_v}= \psi_v^{-1}\circ \widetilde \sigma \circ \psi_v$,  

\item if $\tau(v)\ne v$, then
    ${\tau_{\Gamma}}\restrict{S_v}= \psi_{\tau(v)}^{-1}\circ \sigma_{\val(v)} \circ \psi_v$. 
\end{enumerate}
 Note that the three above conditions imply that  if $\tau_\Gamma$ lifts
$\tau$, then 
$\tau_\Gamma(S_e)=S_{\tau(e)}$ for each edge of $\Gamma$.

We denote by $G$ the subgroup of the group of homeomorphisms of $S_\Gamma$
generated by elements $h:S_\Gamma\to S_\Gamma$ such that
\begin{itemize}
\item $h$ restricts to the identity on
$\displaystyle\bigcup_{v\in\Ve^0(\Gamma)} S_v$;
\item if $e$ is a $\tau$-invariant edge of $\Gamma$, then the restriction
  of $h$ on $S_e$ is isotopic to a power of the Dehn twist along
  $\gamma_e$;
\item if $\tau(e)\ne e$, then the restriction
  of $h$ on $S_e\cup S_{\tau(e)}$ is isotopic to a power of the Dehn twist along
  $\gamma_e- \gamma_{\tau(e)}$.

\end{itemize}

Recall that by definition, two real topological surfaces
 $(S_1,\tau_{1})$ and $(S_2,\tau_{2})$ are isomorphic if there exists
  an orientation-preserving homeomorphism $\theta :S_1\to
 S_2$ such that $\theta\circ \tau_{1}\circ \theta^{-1}= \tau_{2}$. 
In particular, if $(\Gamma,\tau)$ is a real trivalent graph,
and if two lifts $\tau_{\Gamma,1}$ and $\tau_{\Gamma,2}$ of $\tau$ are
conjugated by an element of $G$,  then the real topological surfaces
 $(S_\Gamma,\tau_{\Gamma_1})$ and $(S_\Gamma,\tau_{\Gamma_2})$ are isomorphic.

\begin{definition}
\label{def:realpatchworking}
{\rm A \emph{real structure above} a real trivalent graph $(\Gamma,\tau)$ is a lift of 
 $\tau$ considered up to conjugation by an element of $G$.}
\end{definition}

Clearly, the set of lifts of $\tau$ for a real trivalent graph
$(\Gamma,\tau)$ depends on all the choices
made above. 
However, it follows from Proposition~\ref{lem:patch aff} below that the set of real structures above $(\Gamma,\tau)$ only depends on $(\Gamma,\tau)$.

\begin{lemma}\label{lem:2 patch}
Let  $\mathcal S_2$  be the cylinder $\{z\in\C \:\:   | \:\:
1\le |z|\le 2\}$, 
and 
$\conj$ be the complex conjugation on
$\CC$. 
Denote by $T$ 
 the Dehn twist of $\mathcal S_2$ along the
circle of radius $\frac{3}{2}$ 
 given by $$\begin{array}{cccc}
T: &\mathcal S_2&\longrightarrow& \mathcal S_2
\\ & z &\longmapsto & e^{2i\pi(|z|-1)} z
\end{array}.$$
Then, up to isotopy and conjugation by a homeomorphism 
$\theta :\mathcal S_2\to \mathcal S_2$ restricting to the
identity on $\partial \mathcal S_2$, the only real structures  on
$\mathcal S_2$ restricting to $\conj$ on $\partial \mathcal S_2$ are $\conj$ and
$T\circ \conj$. 
\end{lemma}

The
real part of the two real structures $\conj$ and
$T\circ \conj$ differ
by a ``half'' Dehn twist, see Figure~\ref{fig:real annulus}. 
\begin{figure}[h!]
\begin{center}
\begin{tabular}{ccc}
\includegraphics[width=3.5cm, angle=0]{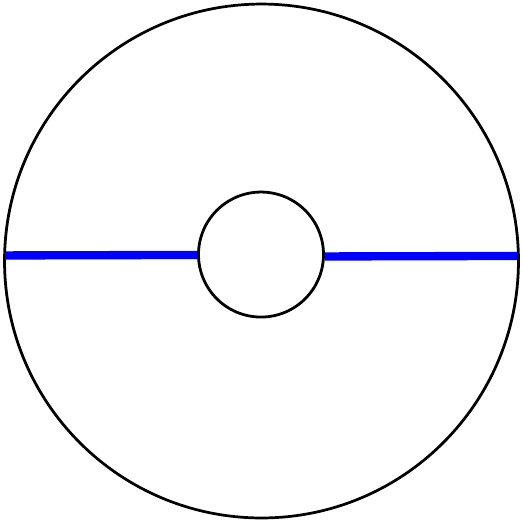}
&\hspace{8ex} 
&\includegraphics[width=3.5cm, angle=0]{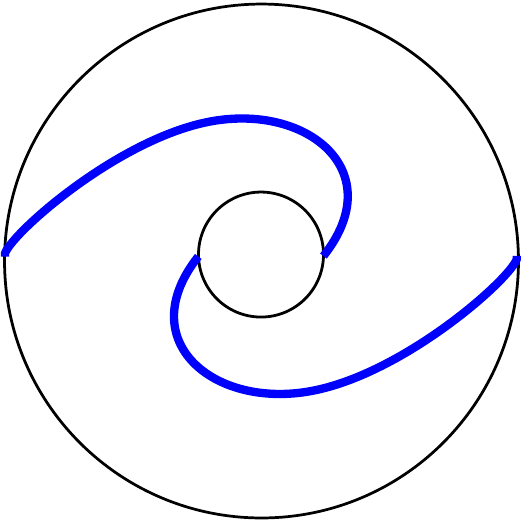}
\\
\\   $z\mapsto \overline z$
&&  $z\mapsto e^{2i\pi(|z|-1)}\overline z$
\end{tabular}
\end{center}
\caption{The real part of two real structures on the annulus $\{z\in\C \:\:   | \:\:
  1\le |z|\le 2\}$} 
\label{fig:real annulus}
\end{figure}
\begin{proof}
Let $\tau_1:\mathcal S_2\to \mathcal S_2$ be an orientation-reversing
continuous involution that  restricts to $\conj$ on $\partial \mathcal
S_2$. 
The map $\tau_1\circ \conj$ is an orientation-preserving
continuous involution that  restricts to the identity on $\partial \mathcal
S_2$. Hence it is, up to a composition by an isotopy restricting to the identity on $\partial \mathcal
S_2$, equal to  $T^k$ for some integer $k$.
We compute $$ T\circ \conj\circ T^{-1}(z)=e^{4i\pi(|z|-1)} \overline
z=T^2\circ \conj. $$
Hence there exists  a homeomorphism $\theta$ as in the
lemma such that 
$\tau_1=\theta\circ \conj\circ \theta^{-1}$ 
if $k$ is even, and such that
$\tau_1=\theta\circ T\circ \conj\circ \theta^{-1}$ 
 if $k$ is odd. 

If there would exist  a homeomorphism $\theta$ as in the
lemma such that $\conj=\theta\circ T\circ \conj\circ \theta^{-1}$, it
 would map the real part of $T\circ \conj$ to the real part of $\conj$.
Since these two real parts differ by a ``half'' Dehn twist,
this is impossible.
\hfill $\Box$
\end{proof}

\begin{lemma}\label{lem:2 patch2}
 Let $\mathcal S'_2$ be the cylinder $\{z\in\C \:\:   | \:\:
\frac{1}{4}\le |z|\le \frac{1}{2}\}$, let  $\mathcal
S''_2$ be the cylinder $\{z\in\C \:\:   | \:\: 
2\le |z|\le 4\}$. 
Denote respectively by $T_1$ and $T_2$  
the Dehn twists of $\mathcal S'_2$ and $\mathcal S''_2$ 
 given by $$\begin{array}{cccc}
T_1: &\mathcal S'_2&\longrightarrow& \mathcal S'_2
 \\ & z &\longmapsto & e^{2i\pi(1- \frac{1}{2|z|})}z
\end{array}
\quad \mbox{and}\quad \begin{array}{cccc}
T_2: &\mathcal S''_2&\longrightarrow& \mathcal S''_2
\\ & z &\longmapsto & e^{2i\pi(\frac{|z|}{2}-1)}z
\end{array}.$$
Then, up to isotopy and conjugation by a homeomorphism 
$\theta :\mathcal S'_2\cup\mathcal S''_2 \to 
\mathcal S'_2\cup\mathcal S''_2 $ restricting to the identity
on $\partial(\mathcal S'_2\cup\mathcal S''_2)$, the only real structures  on
$\mathcal S'_2\cup\mathcal S''_2 $
 restricting to $\frac{1}{\conj}$ on $\partial(\mathcal S'_2\cup\mathcal S''_2)$ are $\frac{1}{\conj}$ and
$T_1\circ T_2\circ \frac{1}{\conj}$. 
\end{lemma}

\begin{proof}
The proof is similar to the proof of Lemma \ref{lem:2 patch}.
Let $\tau_1:\mathcal S'_2\cup\mathcal S''_2 $ be an orientation-reversing
continuous involution that  restricts to $\frac{1}{\conj}$ on 
$\partial(\mathcal S'_2\cup\mathcal S''_2)$.
The map $\tau_1\circ \frac{1}{\conj}$ is an orientation-preserving
continuous involution that  restricts to the identity on
$\partial(\mathcal S'_2\cup\mathcal S''_2)$. 
Hence it is, up to a composition by an isotopy restricting to the
identity on 
$\partial(\mathcal S'_2\cup\mathcal S''_2)$, equal to  $T_1^k\circ
T_2^l$ for some integers $k$ and $l$. 
 We extend the Dehn twists $T_1$ and $T_2$ 
 to $\CC$ by the identity 
respectively outside $\mathcal S'_2$ and $\mathcal S''_2$.
We compute
$$
{\left(T_1^k\circ T_2^{l}\circ  \frac{1}{\conj}\right)}^2(z)=ze^{2i\pi(l-k)(\frac{1}{2|z|}-1)}
$$
if $z\in \mathcal S'_2$ and
$$
{\left(T_1^k\circ T_2^{l}\circ  \frac{1}{\conj}\right)}^2(z)=ze^{2i\pi(l-k)(\frac{|z|}{2}-1)}
$$
if $z\in \mathcal S''_2$. Hence the map $
(T_1^k\circ T_2^{l})\circ  \frac{1}{\conj}$ is an involution on
 $\mathcal S'_2\cup\mathcal S''_2$ if and only if $k=l$.
 We compute
$$
(T_1\circ T_2)\circ \frac{1}{\conj} \circ (T_1^{-1}\circ T_2^{-1})=(T_1^2\circ T_2^{2})\circ \frac{1}{\conj},
$$
and the lemma is proved as Lemma \ref{lem:2 patch}.
\hfill $\Box$
\end{proof}

The following proposition is the main observation that allows one to
compare two different real structures 
 above 
a given real trivalent graph
$(\Gamma,\tau)$, or in other words, to make sense of the difference 
$\tau_{\Gamma,1}-\tau_{\Gamma,2}$ of two real structures in  $\Pi_{(\Gamma,\tau)}$.
Let $\Gamma/\tau$ be the quotient of the graph $\Gamma$ by $\tau$,
i.e.
edges and vertices exchanged by $\tau$ are identified.
\begin{proposition}\label{lem:patch aff}
Let $(\Gamma,\tau)$ be a real trivalent graph. 
Let $e_1,\ldots, e_k$ be the 
 edges in $\Ed^0(\Gamma/\tau)$, 
 Then the set  $\Pi_{(\Gamma,\tau)}$ of real structures above $(\Gamma,\tau)$ is naturally
an affine space with direction 
$$\overrightarrow{\Pi_{(\Gamma,\tau)}}=\Z/2\Z e_1\oplus \Z/2\Z e_2\oplus
\ldots \oplus \Z/2\Z e_{k}.$$
Furthermore the set of fixed points of such a real structure 
$\tau_\Gamma$ is well defined up to isotopy and  Dehn twists on cylinders corresponding
to $\tau$-invariant edges in $\Ed^0(\Gamma)$. 
\end{proposition}
\begin{proof}
Let $e$ be a $\tau$-invariant edge in $\Ed^0(\Gamma)$. It follows from
Lemma \ref{lem:2 patch}  that there exist exactly two possibilities 
for the restriction of a real structure on $S_e$, that are given, up
to isotopy and conjugating by $T$, by $\conj$ and $T\circ \conj$ on the cylinder
$S_e=\{z\in\C \:\:   | \:\:  1\le |z|\le 2\}$.

Similarly, for each pair of cylinders $\{S_e,S_{\tau(e)}\}$ 
with $e\in\Ed^0(\Gamma)$ such that $\tau(e)\neq e$, 
it follows from
Lemma \ref{lem:2 patch2} that there are exactly two possibilities for the restriction of a real
structure on $S_e\cup S_{\tau(e)}$, which differ by a composition
with the Dehn twist along $\gamma_e-\gamma_{\tau(e)}$. 
\hfill $\Box$
\end{proof}

\begin{example}\label{ex:genus2}
{\rm On Figure~\ref{fig:realG2} we show the fixed point set of two different real structures lifting the involution $\tau$ on the real graph of genus 2 depicted on Figure~\ref{fig:twist}a.}
\begin{figure}[h!]
\begin{center}
\begin{tabular}{ccc}
\includegraphics[width=5cm, angle=0]{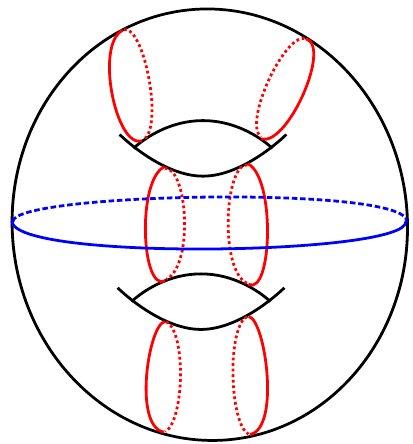}
&\hspace{8ex}
&\includegraphics[width=5cm, angle=0]{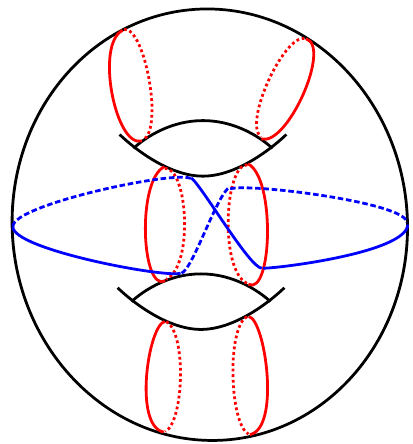}

\end{tabular}
\end{center}
\caption{Two real structures above the same graph $(\Gamma,\tau)$ of genus $2$ of Example~\ref{ex:twist}} 
\label{fig:realG2}
\end{figure}
  
\end{example}

\begin{remark}\label{rem:ribbon}
{\rm In the case when $\tau=\Id$, the data of a real structure
$\tau_\Gamma:S_\Gamma\to S_\Gamma$ above $(\Gamma,\Id)$ is
equivalent to the data of the ribbon structure on
$\Gamma$ given by $S_\Gamma/\tau_\Gamma$. Furthermore, two 
real structures differ along an edge $e$ of $\Gamma$ if and only if the
two ribbon structures differ by a half twist along $e$.}
\end{remark}

\section{Action on $H_1(S_\Gamma;\Z/2\Z)$ induced by a real structure}\label{sec:action}
We identify in this section
a vector subspace $W_{(\Gamma,\tau)}$ of
$\overrightarrow{\Pi_{(\Gamma,\tau)}}$ that characterises all real structures
$\tau_{\Gamma}:S_\Gamma\to S_\Gamma$ inducing the same map
$\tau_{\Gamma *}:H_1(S_\Gamma;\Z/2\Z)\to H_1(S_\Gamma;\Z/2\Z)$.
Next lemma ensures that the action on $H_1(S_\Gamma;\Z/2\Z)$ induced by a
real structure above a real trivalent graph is well defined. 
\begin{lemma}
Let  $(\Gamma,\tau)$ be a real trivalent graph and  $\tau_{\Gamma}:S_\Gamma\to
S_\Gamma$ be a lift of $\tau$. Then
the induced involution $\tau_{\Gamma*}:H_1(S_\Gamma;\Z/2\Z)\to H_1(S_\Gamma;\Z/2\Z)$
only 
 depends on the real structure class of $\tau_\Gamma$.
\end{lemma}
\begin{proof}
It follows from the proof of Lemmas \ref{lem:2 patch} and \ref{lem:2 patch2} 
that any other representative of the real structure class of $\tau_\Gamma$ is
obtained from this latter by a finite sequence of compositions with
either an isotopy or by an even power of a Dehn twist. Since both types of
maps induce a trivial action on $H_1(S_\Gamma;\Z/2\Z)$, the result follows.
\hfill $\Box$
\end{proof}

Given  $(\Gamma,\tau)$    a real trivalent graph, we denote by
$\pi:\Gamma\to \Gamma/\tau$ the quotient map. 
 Recall that by Proposition \ref{lem:patch aff}, an edge in $\Ed^0(\Gamma)$
 defines a vector in $\overrightarrow{\Pi_{(\Gamma,\tau)}}$.
There is a natural bilinear map
$$\begin{array}{cccc}
\mu :& H_1(\Gamma;\Z/2\Z)\times H_1(\Gamma;\Z/2\Z) & \longrightarrow & \overrightarrow{\Pi_{(\Gamma,\tau)}}
\end{array} $$ 
 that associates to two 1-cycles $\alpha$ and $\beta$
 the sum of the vectors in $\overrightarrow{\Pi_{(\Gamma,\tau)}}$ defined by the edges contained 
 in the support of the 1-chain $\pi(\alpha\cap\beta)$.
 We denote by $\langle . , . \rangle$ the standard bilinear form
 on $\overrightarrow{\Pi_{(\Gamma,\tau)}}$ defined by
 $$\left \langle  \sum_{e \in\Ed^0(\Gamma/\tau)} u_e e , \sum_{e \in\Ed^0(\Gamma/\tau)} v_e e \right\rangle =\sum_{e \in\Ed^0(\Gamma/\tau)} u_e v_e.$$
We define the following vector subspace $W_{(\Gamma,\tau)}$ of $\overrightarrow{\Pi_{(\Gamma,\tau)}}$:
$$
W_{(\Gamma,\tau)}=\left\lbrace w\in \overrightarrow{\Pi_{(\Gamma,\tau)}} \:\: \vert \:\: 
\forall \alpha,\beta\in H_1(\Gamma;\Z/2\Z), \:\: \langle w,\mu (\alpha,\beta)\rangle=0\right\rbrace.$$
By definition, we have
$$W_{(\Gamma,\tau)}=\bigcap_{\alpha,\beta\in H_1(\Gamma;\Z/2\Z)} Ker \left(\mu^*(\alpha,\beta)\right),$$
where $\mu^*(\alpha,\beta)$ denotes the linear map dual to the vector 
$\mu(\alpha,\beta)$.

\begin{theorem}\label{thm:main}
Let $(\Gamma,\tau)$ be a real trivalent graph, and let $\tau_{\Gamma, 1}:S_\Gamma\to
S_\Gamma$ and $\tau_{\Gamma, 2}:S_\Gamma\to S_\Gamma$ be two real structures
 above $(\Gamma,\tau)$.  Then we have
$$\tau_{\Gamma, 1 *}=\tau_{\Gamma, 2 *} \Longleftrightarrow\tau_{\Gamma, 1}-\tau_{\Gamma, 2}\in W_{(\Gamma,\tau)}.$$ 
\end{theorem}
\begin{proof}
Clearly, both maps $\tau_{\Gamma, 1 *}$ and $\tau_{\Gamma, 2 *}$ have the same
restriction  to   $H_{1,0}(S_\Gamma;\Z/2\Z)$.
Hence according to
Lemma \ref{lem:hodge dec}, we are left to show that 
$$
\forall \alpha \in H_1(\Gamma;\Z/2\Z),\:\:  
 \tau_{\Gamma, 1 *}(\gamma_\alpha)=\tau_{\Gamma, 2 *}(\gamma_\alpha)  \Longleftrightarrow \tau_{\Gamma, 1}-\tau_{\Gamma, 2}\in W_{(\Gamma,\tau)}.
$$ 
Given a vector $w\in \overrightarrow{\Pi_{(\Gamma,\tau)}}$, we denote by
$\mathcal E(w)$ the set of 
edges $e\in\Ed^0(\Gamma)$ 
such that $\langle w,\pi(e)\rangle=1$.
Choosing other representatives of the real structure class of
$\tau_{\Gamma, 1}$ 
and $\tau_{\Gamma, 2}$ if necessary, 
we may assume that 
these latter
coincide on 
$$\displaystyle S_\Gamma\setminus \bigcup_{e\in \mathcal E(\tau_{\Gamma,
    1}-\tau_{\Gamma, 2})}S_e, $$
 and differ by a composition with a Dehn
twist along $\gamma_e$ for all edges $e$ in 
$\mathcal E(\tau_{\Gamma,1}-\tau_{\Gamma, 2})$.
Hence we have
\begin{equation}\label{equ:diff action}
\tau_{\Gamma, 1 *}(\gamma_\alpha)=\tau_{\Gamma, 2 *}(\gamma_\alpha)+ \sum_{e\in \mathcal
  E(\tau_{\Gamma, 1}-\tau_{\Gamma, 2})\cap \alpha} \gamma_{e}.
\end{equation}

Since the intersection form is non-degenerate on $H_1(S_\Gamma;\Z/2\Z)$, we 
have 
$$\tau_{\Gamma, 1 *}(\gamma_\alpha)-\tau_{\Gamma, 2 *}(\gamma_\alpha)=0
\Longleftrightarrow
\forall \gamma\in H_1(S_\Gamma;\Z/2\Z), \ (\tau_{\Gamma, 1 *}(\gamma_\alpha)-\tau_{\Gamma,
  2 *}(\gamma_\alpha))\cdot \gamma=0,$$
where $a \cdot b\in\Z/2\Z$ stands for the intersection product of two
classes $a$ and $b$ in $H_1(S_\Gamma;\Z/2\Z)$.
Combining Relation $(\ref{equ:diff action})$ and Lemma \ref{lem:hodge
  dec},
 we obtain
$$\tau_{\Gamma, 1 *}(\gamma_\alpha)-\tau_{\Gamma, 2 *}(\gamma_\alpha)=0
\Longleftrightarrow
\forall \beta \in H_1(\Gamma;\Z/2\Z), \ \sum_{e\in \mathcal
  E(\tau_{\Gamma, 1}-\tau_{\Gamma, 2})\cap \alpha} \gamma_{e}\cdot \gamma_\beta=0.$$
By construction we have
$$\sum_{e\in \mathcal
  E(\tau_{\Gamma, 1}-\tau_{\Gamma, 2})\cap \alpha} \gamma_{e}\cdot \gamma_\beta=
\langle \tau_{\Gamma, 1}-\tau_{\Gamma, 2} ,\ \mu( \alpha,\beta)\rangle , $$
so  the result is proved.
\hfill $\Box$
\end{proof}

\begin{example}
{\rm Let us consider the real trivalent graph $(\Gamma,\tau)$ of genus 3 depicted in
Figure~\ref{fig:twist}b, with two 
$\tau$-invariant edges $e_1$ and $e_2$.
There exists 16 real structures  above $(\Gamma,\tau)$, distributed into $8$ 
parallel affine subspaces of  $\Pi_{(\Gamma,\tau)}$ with direction
$W_{(\Gamma,\tau)}=\Z/2\Z \pi(e_1+e_2)$. 
Note that the real part of the real topological surface $(S_\Gamma,\tau_\Gamma)$
is composed 
of two circles for any element $\tau_\Gamma$ of $\Pi_{(\Gamma,\tau)}$. 
Furthermore, all possible real structures  above 
 $(\Gamma,\tau)$ produce exactly two
 real isomorphism types of real topological surfaces:
 $(S_\Gamma,\tau_\Gamma)$ is
  of type\footnote{Recall that a connected real topological surface $(S,\tau)$ is of
  type \RNum{1} if $S\setminus\R S$ is disconnected, and of type \RNum{2} otherwise.} 
\RNum{1} for $4$ out of these $8$ families of real structures, and
 of type \RNum{2} for the $4$ remaining families.} 
\end{example}

\begin{corollary}\label{cor:patch max}
Let $(\Gamma,\tau)$ be a real trivalent graph, and let $\tau_{\Gamma, 1}:S_\Gamma\to
S_\Gamma$ and $\tau_{\Gamma, 2}:S_\Gamma\to S_\Gamma$ be two real structures
 above $(\Gamma,\tau)$. Assume that $(S_\Gamma,\tau_{\Gamma,1})$ is a maximal real
topological  surface. 
Then $(S_\Gamma,\tau_{\Gamma,2})$ is also  maximal  if and only if the
difference $\tau_{\Gamma,1}-\tau_{\Gamma, 2}$ lies in $W_{(\Gamma,\tau)}$.  
\end{corollary}
\begin{proof}
Recall that a real
topological  surface $(S,\tau_S)$ 
with $\R S\ne \emptyset$ is maximal if and only
if $\tau_{S*}=\Id$, see \cite[Proposition 5.4.9]{RisBen90}.
Hence we have $\tau_{\Gamma,1*}=\Id$. Furthermore, since 
 $(S_\Gamma,\tau_{\Gamma,1})$ has a non-empty real part, the trivalent graph $\Gamma$
has at least one $\tau$-invariant edge, which in its turn implies that
any real structure  above 
 $(\Gamma,\tau)$ has a non-empty real part.
 Then, by Theorem~\ref{thm:main},  $\tau_{\Gamma,2*}=\Id$ if and only
 if $\tau_{\Gamma,1}-\tau_{\Gamma, 2} \in W_{(\Gamma,\tau)}$, that is to say $(S_\Gamma,\tau_{\Gamma,2})$ is maximal if and only if   $\tau_{\Gamma,1}-\tau_{\Gamma, 2}$ lies in $W_{(\Gamma,\tau)}$.
\hfill $\Box$
\end{proof}

Next proposition gives an alternative description of the space $W_{(\Gamma,\tau)}$ in
terms of disconnecting edges and disconnecting pairs of edges. 
 Recall that edges are always considered open. 
Given a 
real connected trivalent graph $(\Gamma,\tau)$, 
we define
the following sets:
\begin{itemize}
\item $\Ed^{0,(1)}(\Gamma,\tau)$ is the set of disconnecting
edges of $\Gamma/\tau$ in $\Ed^0(\Gamma/\tau)$, i.e. edges $e\in\Ed^0(\Gamma/\tau)$ such that $\Gamma/\tau\setminus e$ is not connected;
\item $\Ed^{0,(2)}(\Gamma,\tau)$ is the set of pairs $\{\pi(e),\pi(e')\}\subset
  \Ed^0(\Gamma/\tau)$, where $\{e,e'\}\subset  \Ed^0(\Gamma)$ is 
  such that $\pi(e),\pi(e')\notin \Ed^{0,(1)}(\Gamma,\tau)$
and $\Gamma\setminus \{e,e'\}$ is not connected.
\end{itemize}
\begin{proposition}\label{prop:equiv W}
Let $(\Gamma,\tau)$ be a real trivalent graph. Then we have
$$
W_{(\Gamma,\tau)}=\left(\bigoplus_{e\in\Ed^{0,(1)}(\Gamma,\tau)}\Z/2\Z e\right) \oplus \mathrm{Span} \left\lbrace e+e' \mid \{e,e'\} \in \mathrm{Edge}^{0,(2)}(\Gamma,\tau)\right\rbrace.
$$
\end{proposition}
\begin{proof}
Let  us denote by $V$ the $\Z/2\Z$-vector space on the right hand side of
the equality stated in the proposition. Hence we want to show that
$W_{(\Gamma,\tau)}=V$.

\medskip
\textbf{Step 1: reduction to the case $\tau=\Id$.}
The map $\pi:\Gamma\to \Gamma/\tau$ induces a
linear map $\widetilde \pi:\overrightarrow{\Pi_{(\Gamma,\Id)}}\to\overrightarrow{\Pi_{(\Gamma,\tau)}}$.
 In order to avoid confusion, we denote by 
$\mu_\Gamma$ the map $H_1(\Gamma;\Z/2\Z)\times H_1(\Gamma;\Z/2\Z)\to
\overrightarrow{\Pi_{(\Gamma,\Id)}}$, 
 and
we keep the notation  
$\mu : H_1(\Gamma;\Z/2\Z)\times H_1(\Gamma;\Z/2\Z) \to
\overrightarrow{\Pi_{(\Gamma,\tau)}}$. 
 Recall that
$$W_{(\Gamma,\Id)}=\left\lbrace w\in \overrightarrow{\Pi_{(\Gamma,\Id)}} \:\: \vert \:\: 
\forall \alpha,\beta\in H_1(\Gamma;\Z/2\Z), \:\: \langle w,\mu_\Gamma (\alpha,\beta)\rangle=0\right\rbrace.$$

Let $\overrightarrow{\Pi_{(\Gamma,\Id)}}^{sym} \subset \overrightarrow{\Pi_{(\Gamma,\Id)}}$ be the
subvector space consisting of vectors
$w\in\overrightarrow{\Pi_{(\Gamma,\Id)}}$ fixed by the
natural
map induced by $\tau$ on $\overrightarrow{\Pi_{(\Gamma,\Id)}}$. Note
 that the difference between two real structures above $(\Gamma,\tau)$
can always be expressed as such an element.
Given $w\in\overrightarrow{\Pi_{(\Gamma,\Id)}}^{sym}$, 
we write $w=w^{fix}+w^{even}$ where $w^{fix}$ is supported by the edges of $\Gamma$ fixed by $\tau$ and $w^{even}$ is supported by pairs of edges exchanged by $\tau$.
Define the surjective linear map $\bar \pi:\overrightarrow{\Pi_{(\Gamma,\Id)}}^{sym}\to
\overrightarrow{\Pi_{(\Gamma,\tau)}}$ by $\bar \pi(w)= \widetilde \pi
(w^{fix}) + \widetilde \pi (w^{1})$, where $w^{even}=w^1+\tau(w^1)$
and the supports of $w^1$ and $\tau(w^1)$ are disjoint.
The vector $w^1$ is not unique, however $\tilde{\pi}(w^1)$  does
clearly not depend on the particular choice of $w^1$. 

\medskip
Let us prove that 
$
W_{(\Gamma,\tau)}=\bar\pi\left( W_{(\Gamma,\Id)}\cap\overrightarrow{\Pi_{(\Gamma,\Id)}}^{sym}\right).
$

Let $w\in\overrightarrow{\Pi_{(\Gamma,\Id)}}^{sym}$ and let
$(\alpha,\beta)\in H_1(\Gamma;\Z/2\Z)\times H_1(\Gamma;\Z/2\Z)$. Write
$\mu_\Gamma (\alpha,\beta)= \mu_\Gamma^{fix} + \mu_\Gamma^{even} +
\mu_\Gamma^{odd}$,
where $\mu_\Gamma^{odd}$ is supported by edges
$e\in\alpha\cap\beta$ such that $\tau(e)\notin \alpha\cap\beta$
and  $\mu_\Gamma^{fix}$ and $\mu_\Gamma^{even}$ are defined as above (in general,  the decomposition is not unique). 
We obviously have
$$\langle w,\mu_\Gamma (\alpha,\beta)\rangle =
 \langle w^{fix},\mu_\Gamma^{fix}\rangle +  \langle w^{even},\mu_\Gamma^{odd}\rangle. $$
Similarily, we have $\mu (\alpha,\beta)= \widetilde
\pi(\mu_\Gamma^{fix}) + \widetilde \pi(\mu_\Gamma^{odd})$,
and
$$\langle \bar\pi(w),\mu (\alpha,\beta)\rangle = \langle \bar\pi(w^{fix}), \widetilde
\pi(\mu_\Gamma^{fix})\rangle +
\langle \bar\pi(w^{even}),\widetilde\pi(\mu_\Gamma^{odd})\rangle. $$
Thus  $\langle w,\mu_\Gamma (\alpha,\beta)\rangle = \langle \bar\pi(w),\mu
(\alpha,\beta)\rangle$,
which implies that
$
W_{\Gamma,\tau}=\bar\pi\left( W_{\Gamma,\Id}\cap\overrightarrow{\Pi_{(\Gamma,\Id)}}^{sym}\right)$
as announced.

\medskip
Assuming that Proposition~\ref{prop:equiv W} is known in the case of $\tau=\Id$ and computing
$ W_{(\Gamma,\Id)}\cap\overrightarrow{\Pi_{(\Gamma,\Id)}}^{sym}$, it is easy to see that it remains
to show that
$$
\Ed^{0,(1)}(\Gamma,\tau)=\pi\left(\Ed^{0,(1)}(\Gamma,\Id)\right) \bigcup
\pi\left(  \left\{\{e,e'\}\subset \Ed^{0,(2)}(\Gamma,\Id)\, |\ \tau(e)=e',
 \textrm{ and } \tau(e')=e\right\}  \right)$$
and 
\begin{eqnarray*}
\Ed^{0,(2)}(\Gamma,\tau)&=&\pi\left(\left\{\{e,e'\}\subset \Ed^{0,(2)}(\Gamma,\Id)\, |\ \tau(e)\neq e'\right\}\right)\\
  &=&\pi\left(\left\{\{e,e'\}\subset \Ed^{0,(2)}(\Gamma,\Id)\, |\ \tau(e)=e
\textrm{ and } \tau(e')=e'\right\}\right)
\\ &
&\cup\ 
\pi\left(\left\{\{e,e'\}\subset \Ed^{0,(2)}(\Gamma,\Id)\, |\ \tau(e)\neq e,\>  \tau(e')\neq e'
\textrm{ and } \tau(e)\neq e'\right\}\right).\notag
\end{eqnarray*}
These equalities follows from the fact that  
 the quotient map $\pi:\Gamma\to \Gamma/\tau$ induces a
surjective map $\pi_*:H_1(\Gamma;\Z/2\Z)\to H_1(\Gamma/\tau;\Z/2\Z)$.

\bigskip
So from now on, we assume that $\tau=\Id$, and we use the shorter
notation
$$W_\Gamma=W_{(\Gamma,\tau)},\qquad \Ed^{0,(1)}(\Gamma)= \Ed^{0,(1)}(\Gamma,\Id),\quad \mbox{and}\quad
\Ed^{0,(2)}(\Gamma)= \Ed^{0,(2)}(\Gamma,\Id).$$

\medskip

\textbf{Step 2:  $V$ is contained in $W_{\Gamma}$.} 
An edge in $\Ed^{0,(1)}(\Gamma)$ is not
contained 
 in the support of any cycle in $H_1(\Gamma;\Z/2\Z)$. In
particular we have
$$\bigoplus_{e\in\Ed^{0,(1)}(\Gamma)}\Z/2\Z e\subset W_{\Gamma}.$$

We claim that a pair of edges $\{e,e'\}\subset\Ed^0(\Gamma)$ is contained in
 $\Ed^{0,(2)}(\Gamma)$ if and only if the two
following conditions are satisfied:
\begin{enumerate}
\item there exists a cycle $\gamma\in H_1(\Gamma;\Z/2\Z)$ 
containing both  $e$ and $e'$;
\item for any theta subgraph $\Theta$
of $\Gamma$ containing $e$ and $e'$, the graph $\Theta\setminus \{e,e'\}$ is not connected,
see Figure \ref{fig:graphs}a and b. 
\end{enumerate}
Indeed, condition $(1)$ is equivalent 
 to the fact that neither $e$ nor $e'$
is in $ \Ed^{0,(1)}(\Gamma)$.
Next, a pair $\{e,e'\}$ satisfying 
condition $(1)$
is in  $\Ed^{0,(2)}(\Gamma)$ if and only if
any path in $\Gamma\setminus\{e\}$ joining the two vertices adjacent
to $e$ contains $e'$. This is equivalent to condition $(2)$.

\begin{figure}[h!]
\begin{center}
\begin{tabular}{ccc}
\includegraphics[width=3cm, angle=0]{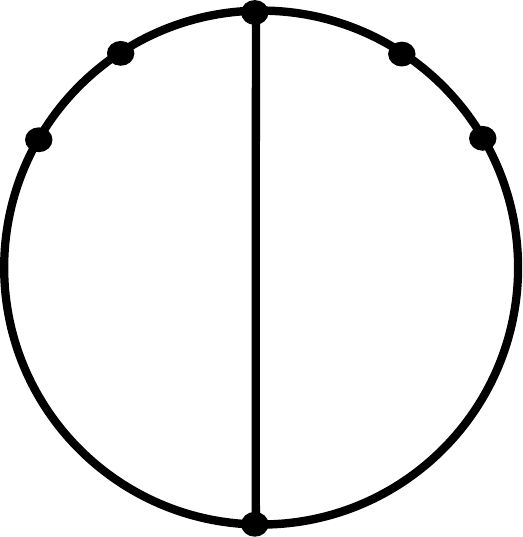}
\put(-85,75){$e$}
\put(-10,75){$e'$}
& \hspace{15ex}
&\includegraphics[width=3cm, angle=0]{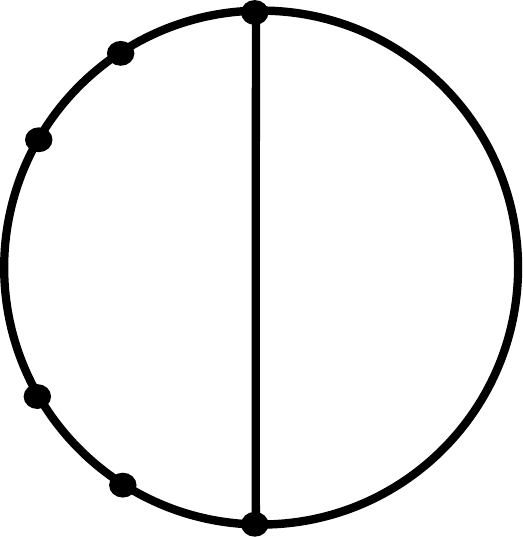}
\put(-85,75){$e$}
\put(-85,10){$e'$}
\\
\\ a) $\Theta\setminus \{e,e'\}$ connected 
&&  b) $\Theta\setminus \{e,e'\}$ disconnected 
\end{tabular}
\end{center}
\caption{} 
\label{fig:graphs}
\end{figure}
It follows from the above claim that
 for any  $\{e,e'\}\in
\Ed^{0,(2)}(\Gamma)$ and any 
$\gamma\in H_1(\Gamma;\Z/2\Z)$, we have
\begin{equation*}\label{equ:same cycle}
  e\subset \gamma \Longleftrightarrow  e'\subset \gamma.
  \end{equation*}
This implies that
$$ \mathrm{Span}\left\lbrace e+e' \mid \{e,e'\}\in \Ed^{0,(2)}(\Gamma)\right\rbrace \subset W_{\Gamma}$$
and so $V\subset W_{\Gamma}$.

\medskip

We define the following vector subspace of  $\overrightarrow{\Pi_{(\Gamma,\tau)}}$:$$U= \left(\bigoplus_{e\in\Ed^{0,(1)}(\Gamma,\tau)}\Z/2\Z e\right)  \oplus \mathrm{Span} \left\lbrace e \mid \exists e' \mbox{ with }\{e,e'\}\in \Ed^{0,(2)}(\Gamma,\tau)\right\rbrace.$$ 

\textbf{Step 3: $W_{\Gamma}$ is contained in $U$.} 
Let $e\notin U$. It means exactly that $e\in\Ed^0(\Gamma)\setminus \Ed^{0,(1)}(\Gamma)$ and that $\Gamma\setminus\{e,e'\}$ is connected for
any $e'\in\Ed^0(\Gamma)\setminus \Ed^{0,(1)}(\Gamma)$. Denote by $v$ and $v'$ the two 
vertices of $\Gamma$ adjacent to $e$.
By Menger's Theorem~\cite[Theorem 3.6.11]{BaRa12}, there exist two paths
$c_1$ and $c_2$ in $\Gamma\setminus \{e\}$ joining $v$ and $v'$, and whose
intersection is reduced to $\{v,v'\}$. Hence $e$ is the only edge
common to 
the two cycles 
$c_1\cup e$ and $c_2\cup e$ in $H_1(\Gamma;\Z/2\Z)$. By definition of $W_{\Gamma}$,
we have $\langle w,e\rangle=0$ for any $w\in W_{\Gamma}$, and so $e\notin W_{\Gamma}$. We conclude that $W_{\Gamma}\subset U$.

\medskip
\textbf{Step 4:  $W_{\Gamma}$ is contained in  $V$.}
 By Step 3 above, an element $\widetilde w$ of $W_{\Gamma}$ can be written in the
(non-unique)  form
$$\widetilde w=w_0+e_1+\ldots +e_k +(e_{k+1} + e'_{k+1}) + \ldots + (e_m + e'_m), $$
where
\begin{itemize}
  \item $\displaystyle
    w_0\in\bigoplus_{e\in\Ed^{0,(1)}(\Gamma)}\Z/2\Z e$;
    \\
\item $\{e_{i},e'_{i}\}\in \Ed^{0,(2)}(\Gamma)$ if $i\in\{k+1,\ldots,m\}$;
  \item for any $i\in\{1,\ldots,k\}$, there exists an edge $e'_i$ of
    $\Gamma$ such that $\{e_i,e'_i\}\in \Ed^{0,(2)}(\Gamma)$;
   
  \item $e_i\ne e_j$ and $\{e_i,e_j\}\notin \Ed^{0,(2)}(\Gamma)$ for any pair
    $\{i,j\}\subset \{1,\ldots,k\}$.
        
\end{itemize}
Recall that $V$ is a sub-vector space of $W_{\Gamma}$ by Step 2, hence
$w=w_0+e_1+\ldots +e_k$ is an element of  $W_{\Gamma}$. If $k=0$, then
$w$ is in $V$, and so is $\widetilde w$. Assume now that $k>0$. In
this case $w$ is an
element of $W_\Gamma$ such that for any edge $e\in\Ed^0(\Gamma)\setminus
\Ed^{0,(1)}(\Gamma)$ with $\langle w,e\rangle =1$, we have 
$\langle w,e' \rangle=0$ for any edge $e'$ with
$\{e,e'\}\in \Ed^{0,(2)}(\Gamma)$. The rest of the proof consists in 
proving by contradiction that such an element cannot exist.

Let $\alpha\in H_1(\Gamma;\Z/2\Z)$ be
a cycle containing an edge in the support of $w$. Let us denote by
$e_1,\ldots,e_k$ all edges belonging to the intersection of the support of $w$ and $\alpha$, enumerated in a cyclic order induced
by $\alpha$. By definition of $W_{\Gamma}$, we have 
$\langle w,\mu(\alpha,\alpha)\rangle=0$, i.e. $k$ is even. 
Let us also denote by $u_i$ the connected component of
$\alpha\setminus\{e_1,\ldots, e_k\}$ adjacent to the edges $e_{i}$ and $e_{i+1}$ (where the indices are taken modulo $k$).
By assumption on $w$, for each $i$ there exists a path $c_i$ in
$\Gamma$ joining
the two connected components of $\alpha\setminus\{e_i,e_{i+1}\}$, and
such that $c_i\cap\alpha$ is reduced to the two endpoints of
$c_i$. Denote by
 $\sigma(i)\in \Z/k\Z$  the  integer such that
 $u_i\cup u_{\sigma(i)}$ contains the two endpoints of $c_i$.
Hence 
$\sigma:\Z/k\Z\to \Z/k\Z$ is an involution with no
fixed points.
The two endpoints of $c_i$ divide $\alpha$ into two connected
components. Let $u$ be one of these latter. Since $w\in W_{\Gamma}$, we have
$\langle w,\mu(\alpha, u\cup c_i))\rangle=0$, that is to say
 $\sigma(i)=i \mod 2$ (see Figure~\ref{fig:graphs2}a). 
\begin{figure}[h!]
\begin{center}
\begin{tabular}{ccc}
\includegraphics[width=3.5cm, angle=0]{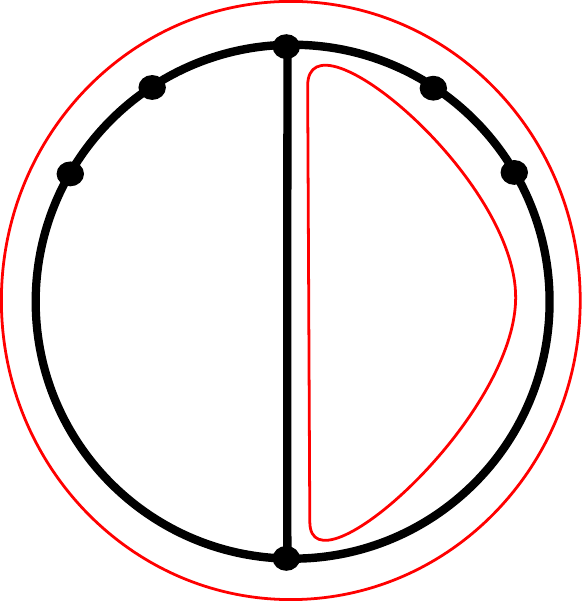}
\put(-98,88){$e_1$}
\put(-10,88){$e_2$}
\put(-98,10){$\alpha$}
\put(-65,50){$c_i$}
\put(-42,50){$c_i\cup u$}
& \hspace{15ex}
&\includegraphics[width=3.5cm, angle=0]{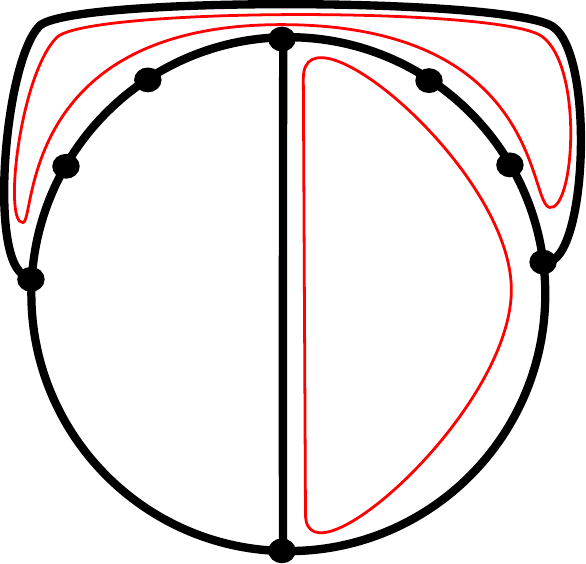}
\put(-35,50){$\beta_1$}
\put(-15,80){$\beta_2$}
\\
\\  a) && b)
\end{tabular}
\end{center}
\caption{} 
\label{fig:graphs2}
\end{figure}
Furthermore, we claim that for any $j$, we have
$$
j\in\{i+1,\ldots,\sigma(i)-1\}\Longleftrightarrow
\sigma(j)\in\{i+1,\ldots,\sigma(i)-1\}.
$$
It is enough to prove the claim in the case when $j=i+1$. If
$\sigma(i+1)\notin\{i+1,\ldots,\sigma(i)-1\}$,  one easily constructs
two cycles $\beta_1$ and $\beta_2$ in $H_1(\Gamma;\Z/2\Z)$ such that
 $\langle w,\mu(\beta_1, \beta_2))\rangle=1$ (see Figure~\ref{fig:graphs2}b), which
contradicts that $w\in W_{\Gamma}$.

Hence the map $\sigma$ induces  an involution on 
$\{2,\ldots,\sigma(1)-1\}$ with no
fixed points. However the cardinal of this latter set is odd, so such
a fixed-point free involution cannot exist.
\hfill $\Box$
\end{proof}

Recall that a graph is said to be \emph{$k$-edge connected} if it
remains connected after removing any set of $l<k$ edges.
\begin{corollary}
Let $(\Gamma,\tau)$ be a real trivalent graph. If 
$\Gamma/\tau$ is $3$-edge connected, then $W_{(\Gamma,\tau)}$ is the trivial vector
space. 
In particular there exists at most one real structure $\tau_\Gamma$
above $(\Gamma,\tau)$ for which $(S_\Gamma,\tau_\Gamma)$ is maximal.
\end{corollary}

\section{Haas' Theorem}\label{sec:patch&haas}

Here we explain how the  results from the previous section
specialise to Haas' Theorem
in
 the particular case of non-singular tropical curves in $\R^2$.
We first give
in  Section \ref{sec:viro}  a tropical formulation of Viro's
combinatorial 
patchworking, and state  Haas' Theorem classifying combinatorial
patchworkings  producing $M$-curves. We prove this latter in
Section \ref{sec:haas}.

We assume that the reader has a certain acquaintance with
tropical geometry. We refer to \cite{BIMS15} for an
 introduction to 
tropical geometry at the level needed here, as well as a more detailed exposition of combinatorial
patchworking and Haas' Theorem. 

\subsection{Combinatorial patchworking and Haas' Theorem}\label{sec:viro}

Here we present the reformulation of the particular case of
unimodular combinatorial patchworking
in
terms of twist-admissible
sets of edges of a non-singular plane
tropical curve given in \cite{BIMS15}.
\medskip
If $e$ is an edge of a tropical curve in $\R^2$, we denote by
$(x_e,y_e)\in\Z^2$ a primitive  direction vector of the line
supporting $e$
(note that $(x_e,y_e)$ is well defined up to sign, however this does not play
a role in what follows). 

\begin{definition}
{\rm Let $C$ be a non-singular tropical curve in $\R^2$. 
A subset $T$ of $\Ed^0(C)$ is called {\em twist-admissible}
if  for any cycle
$\gamma$ of $C$, we have
\begin{equation}\label{Men}
\sum_{e\in\gamma\cap T} (x_e,y_e)=0\quad \text{mod }2.
\end{equation}}
\end{definition} 

Given $T$ a twist-admissible subset of edges of a non-singular
tropical curve $C$ in $\R^2$,  perform the following operations:

\begin{itemize}
\item[(a)] at each vertex of $C$,  draw three arcs as depicted in Figure
  \ref{fig:patch gen}a;

\item[(b)] for each  edge $e\in \Ed^0(C)$ adjacent to the vertices $v$ and $v'$,
 join the two corresponding arcs at $v$  to the
  corresponding ones for $v'$ in the following way: if $e\notin T$, then join these arcs
 as depicted in Figure
  \ref{fig:patch gen}b; if $e\in T$, then join these arcs
 as depicted in Figure
  \ref{fig:patch gen}c; 
denote by $\P$ the  obtained collection of arcs;

\item[(c)]  choose arbitrarily an arc of $\P$ 
and a pair of signs for it; 

\item[(d)]  associate pairs of signs to all arcs of $\P$ using the
  following rule: given  
  $e\in \Ed(C)$
the pairs
of signs of the two arcs of $\P$ corresponding to $e$ 
differ by a factor $((-1)^{x_e},(-1)^{y_e})$, see Figure
\ref{fig:patch gen}b and c;
(Note that the compatibility condition
\eqref{Men}
 precisely means that this rule is consistent.)
\begin{figure}[h]
\centering
\begin{tabular}{ccc}
\includegraphics[height=4cm,
  angle=0]{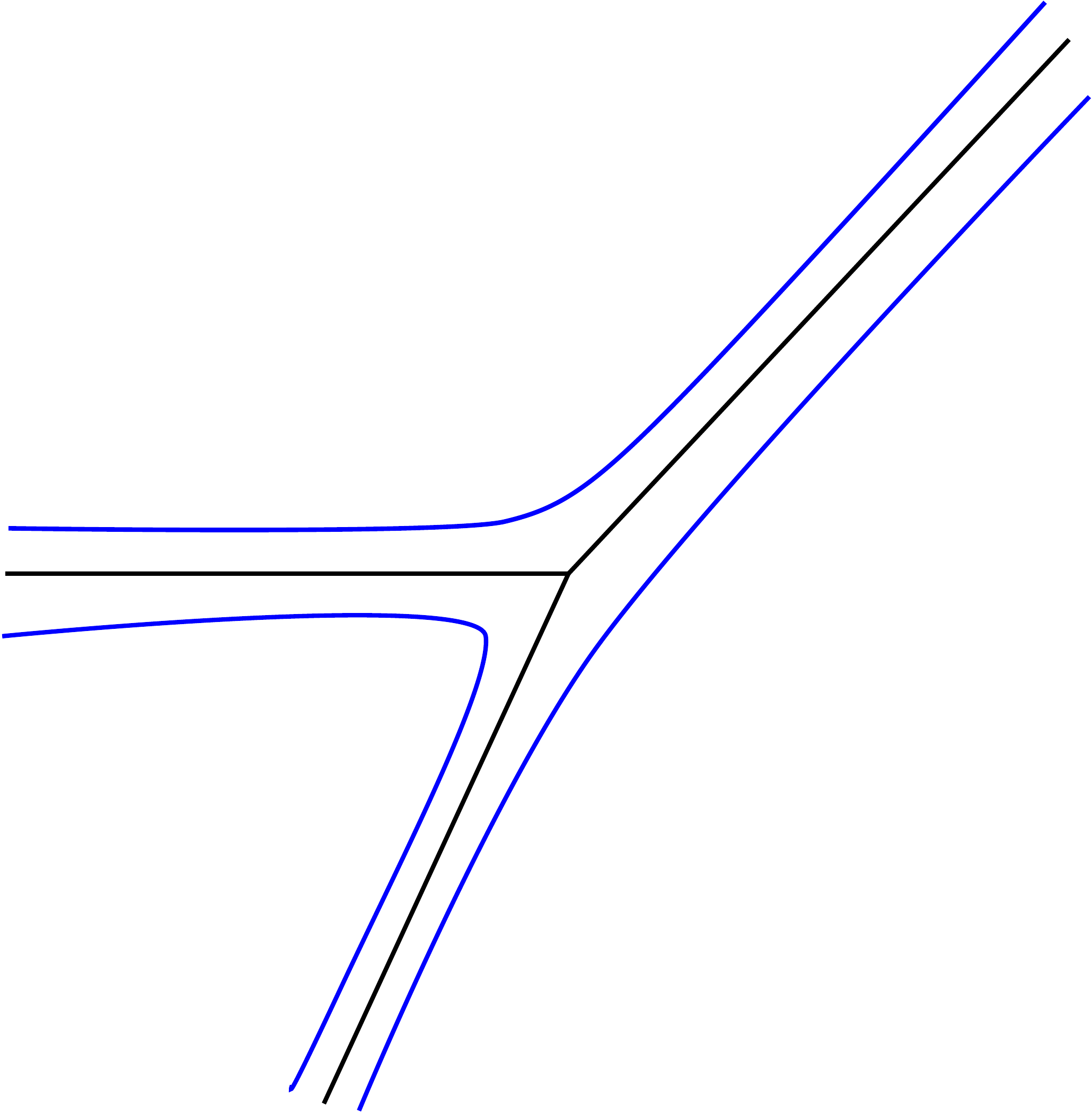} \hspace{2ex} &
\includegraphics[height=5cm, angle=0]{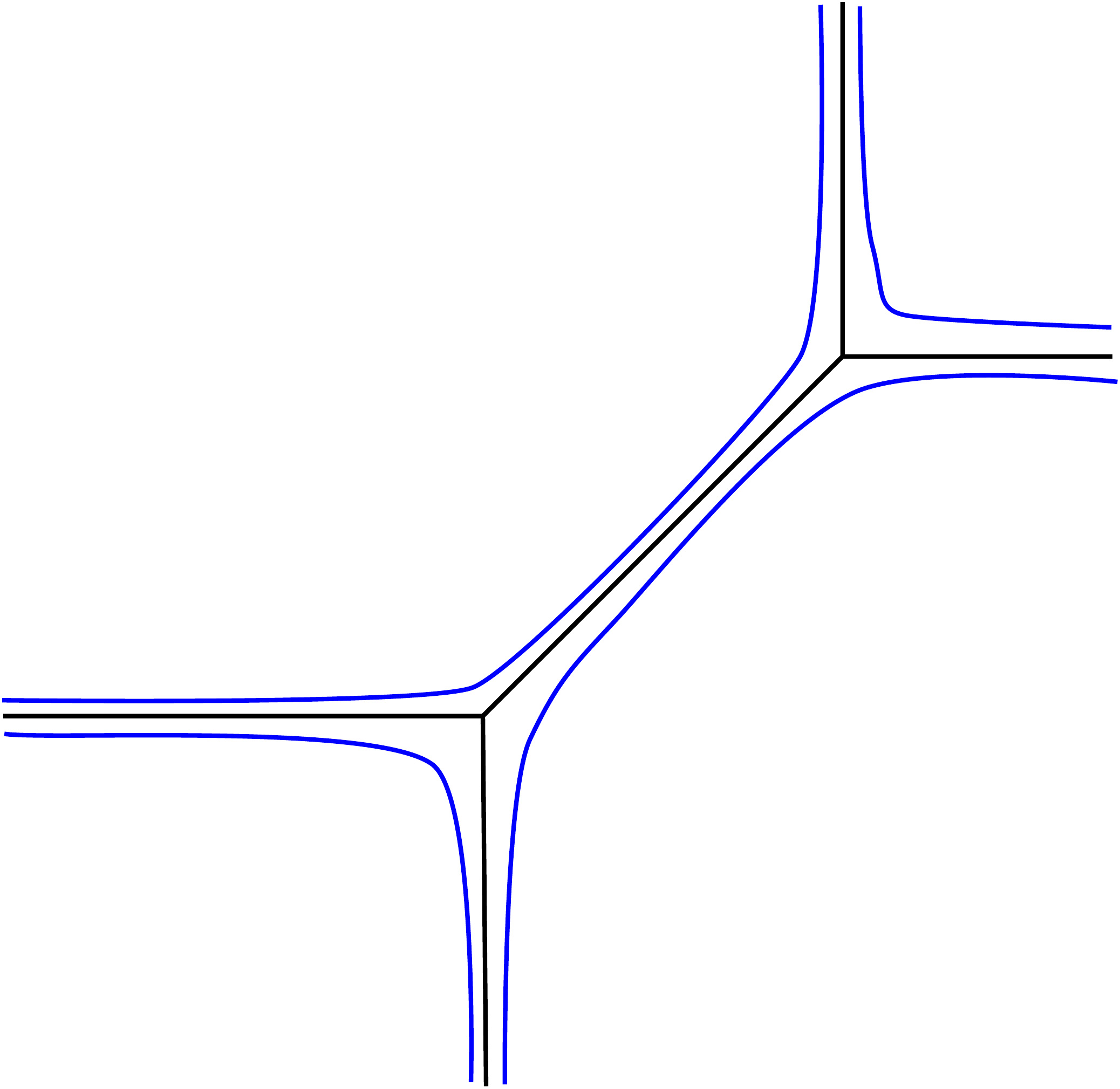}
\put(-100, 65){\tiny{$(\epsilon_1,\epsilon_2)$}}
 \put(-70, 45){\tiny{$((-1)^{x_e}\epsilon_1,(-1)^{y_e}\epsilon_2)$}} \hspace{2ex}
&\includegraphics[height=5cm, angle=0]{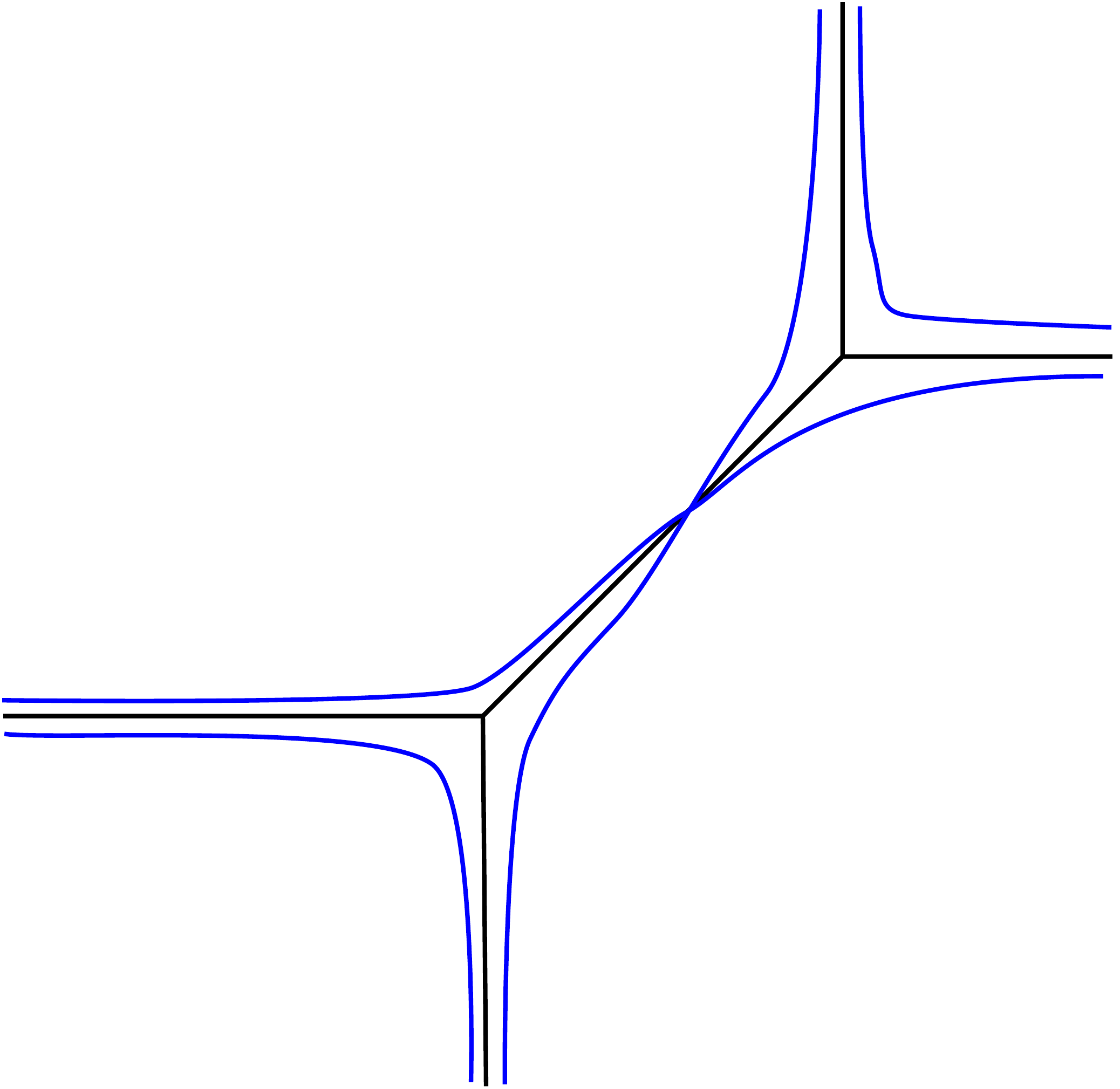}
\put(-100, 65){\tiny{$(\epsilon_1,\epsilon_2)$}}
  \put(-70, 45){\tiny{$((-1)^{x_e}\epsilon_1,(-1)^{y_e}\epsilon_2)$}}
 \\
\\  a) \hspace{2ex} & b) $e \notin T$ \hspace{2ex} & c) $e \in T$
\end{tabular}
\caption{A patchworking of a non-singular tropical curve}\label{fig:patch gen}
\end{figure}

\item[(e)] map each arc $A$ of $\P$ to $(\RR^\times)^2$
by $(x,y)\mapsto (\epsilon_1 \exp(x), \epsilon_2 \exp(y))$, where 
$(\epsilon_1, \epsilon_2)$ is the pair of signs  
associated to $A$. Denote by $C_T$ the  curve in $(\RR^\times)^2$
which is the union of these images over all arcs
of $\P$.

\end{itemize}

Note that all possible choices at step (c) above produce the same
curve $C_T$ up to the action of
$(\Z/2\Z)^2$
by axial symmetries
$(z,w)\mapsto (\pm z,\pm w)$.
 Viro combinatorial patchworking Theorem \cite{V9}
may be reformulated in terms of twist-admissible
sets as follows.
\begin{theorem}[Viro]
\label{thm:viro}
Let $T$ be a twist-admissible subset of edges of a non-singular
tropical curve $C$ in $\R^2$. Then there exists a real algebraic curve
in $(\C^\times)^2$ with the same Newton polygon as $C$, and whose real 
 part in $(\R^\times)^2$ is isotopic to $C_T$.
\end{theorem} 

\begin{remark}
{\rm  It is 
  possible 
to produce an equation for the real algebraic curve whose existence is
attested by Theorem \ref{thm:viro}, see~{\cite[Remark 3.9]{BIMS15}}.}
\end{remark}

\begin{example}
{\rm One may choose $T$ to be empty as 
the empty set clearly satisfies 
\eqref{Men}. This corresponds to Harnack patchworking mentioned
in the introduction.
The resulting curve corresponds to the
construction of  \emph{simple Harnack curves} described in \cite{Mik11}
via Harnack distribution of signs, see \cite{IV2}. Furthermore, the
isotopy type of $C_T$ in $(\R^\times)^2$, up to axial symmetries, only depends on the Newton
polygon of $C$, see Proposition \ref{prop: harnack signs}.}
\end{example}

 \begin{example}
{\rm Let us consider  the non-singular tropical  curve $C$ of degree 6 depicted on
Figure \ref{fig:patch sextic}.
We  equip $C$ with 
 three different twist-admissible collection of edges in Figures
 \ref{fig:patch sextic1}, \ref{fig:patch gudkov}, and \ref{fig:patch hilbert}.
In each case we depict 
 the isotopy types of the 
real part of the corresponding 
real algebraic curve in both $(\RR^\times)^2$ and $\RP^2$. Note that
these are the only isotopy types of maximal real sextics in $\RP^2$.}
\begin{figure}[h]
\centering
\begin{tabular}{c}
\includegraphics[height=5cm, angle=0]{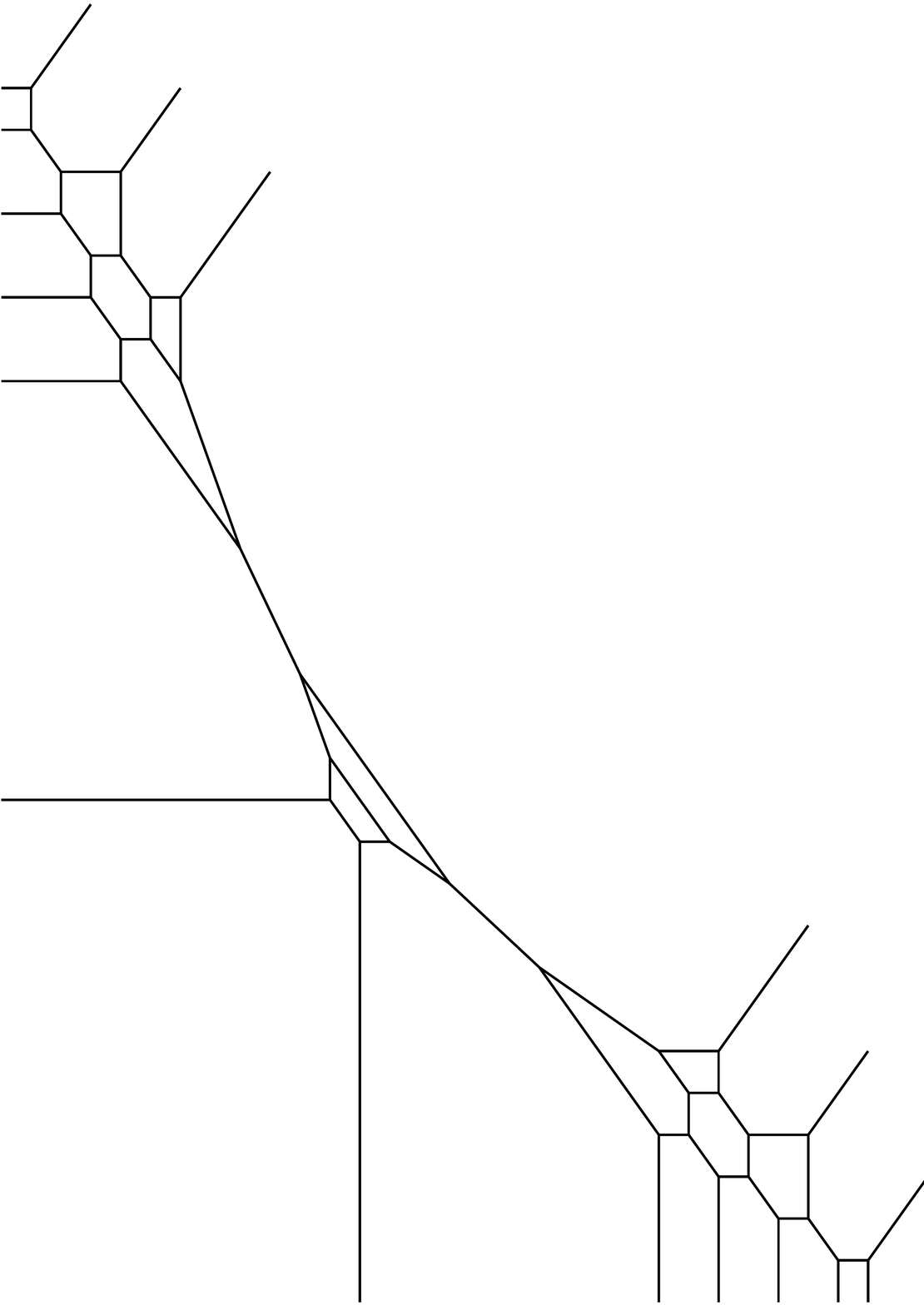}
\end{tabular}
\caption{A tropical sextic}\label{fig:patch sextic}
\end{figure}

\begin{figure}[h]
\centering
\begin{tabular}{cccc}
\includegraphics[height=7cm, angle=0]{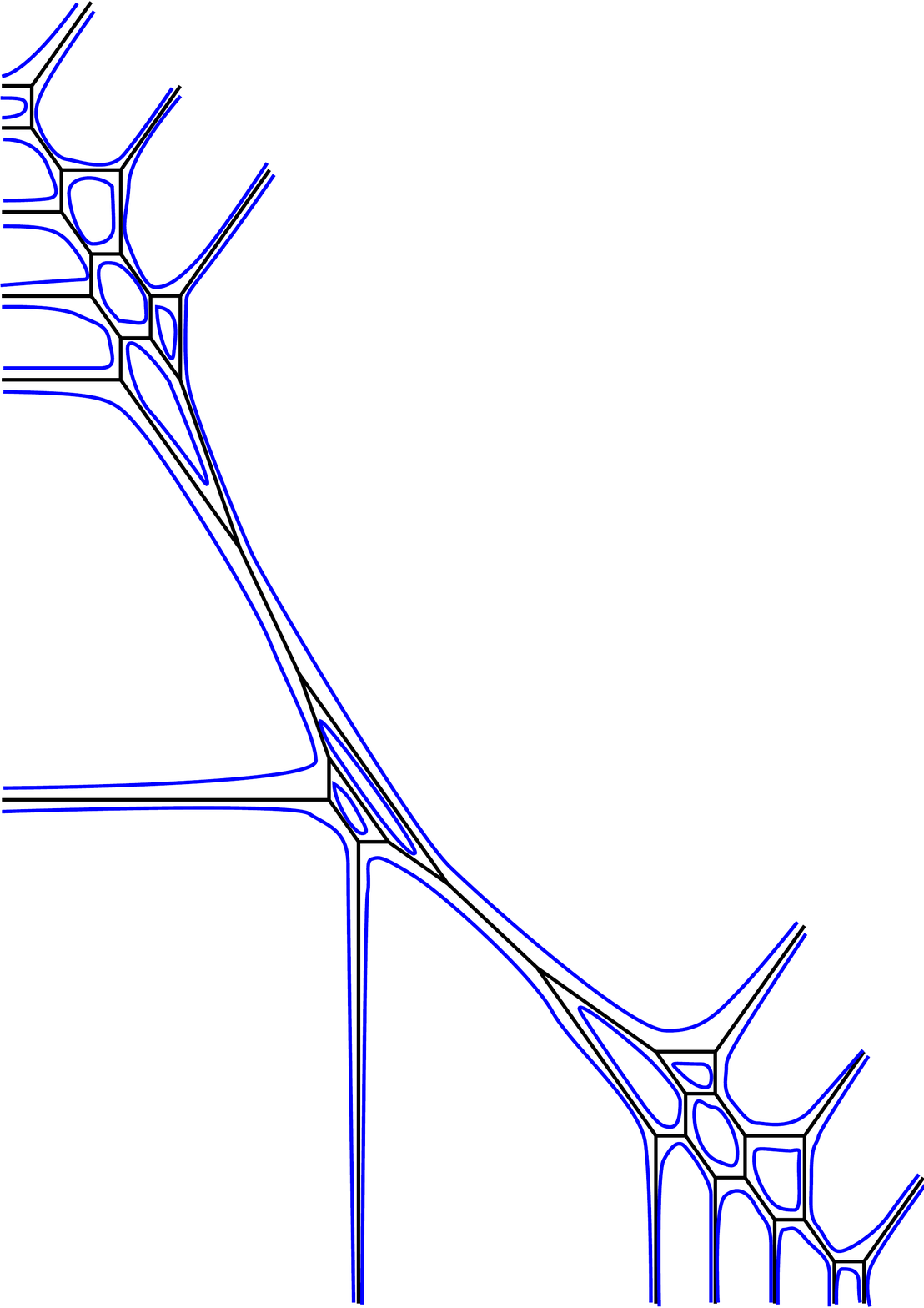}& 
\includegraphics[height=7cm, angle=0]{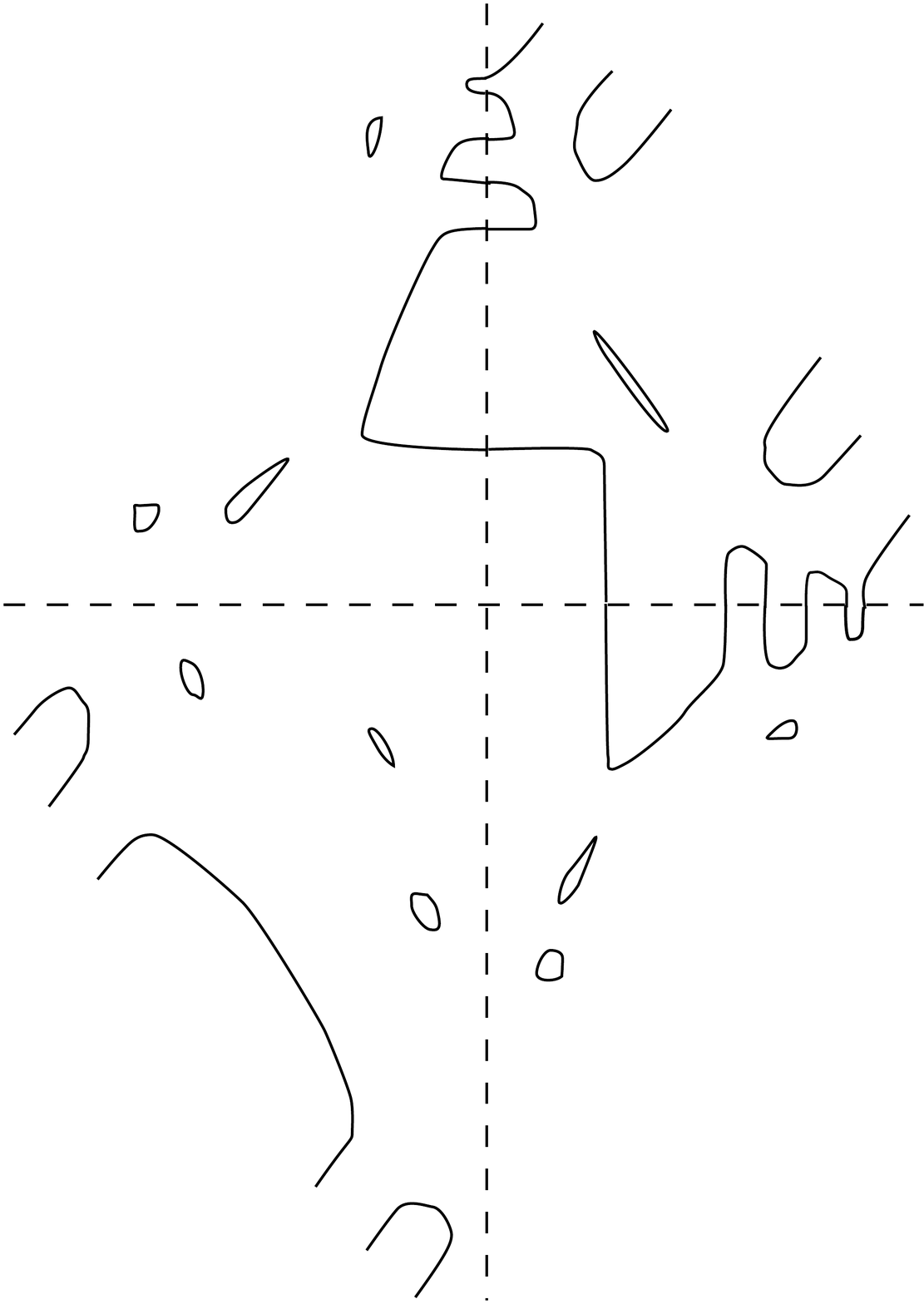}& 
\includegraphics[height=7cm, angle=0]{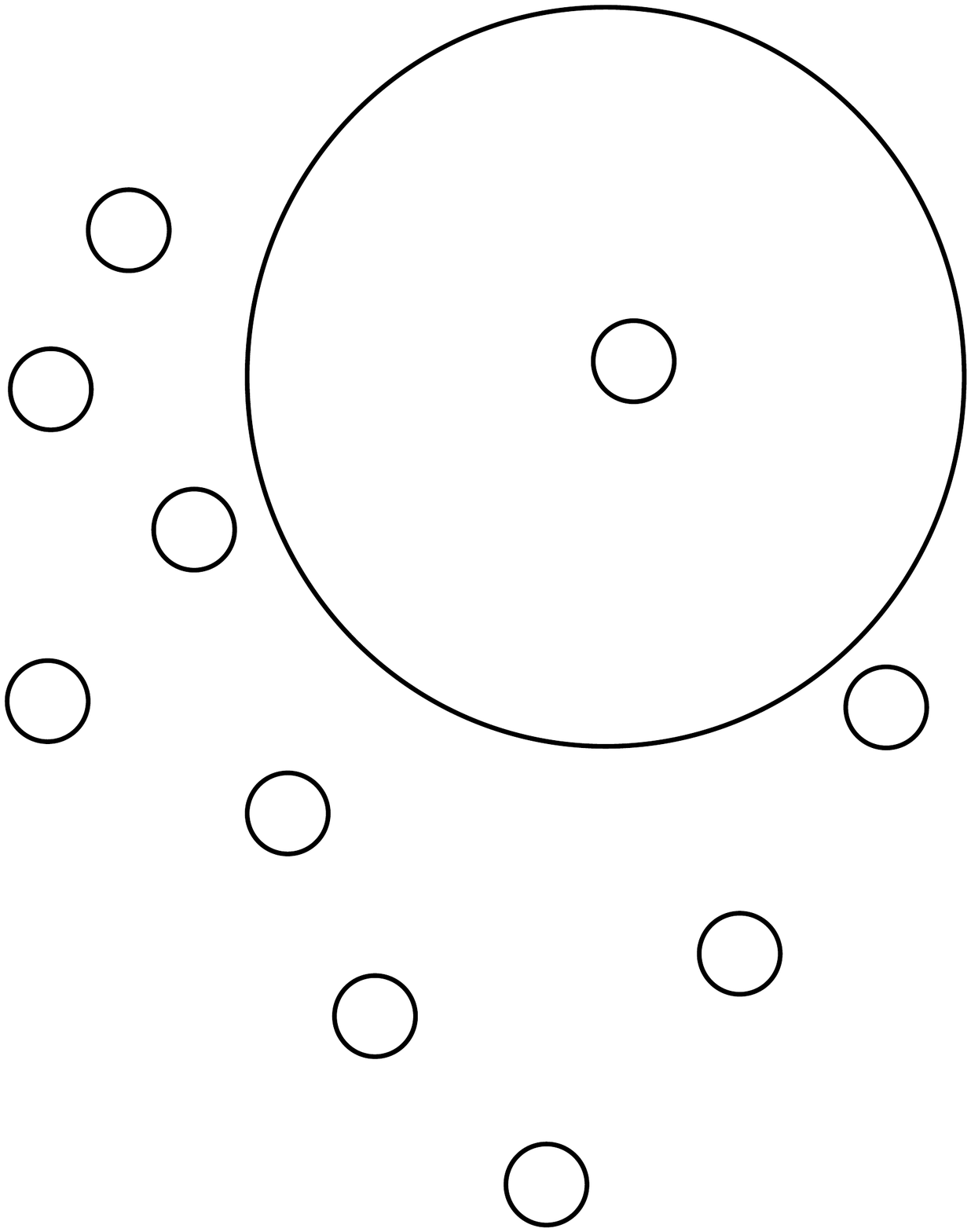}

\end{tabular}
\caption{Simple Harnack sextic}\label{fig:patch sextic1}
\end{figure}

\begin{figure}[h]
\centering
\begin{tabular}{ccc}
\includegraphics[height=7cm, angle=0]{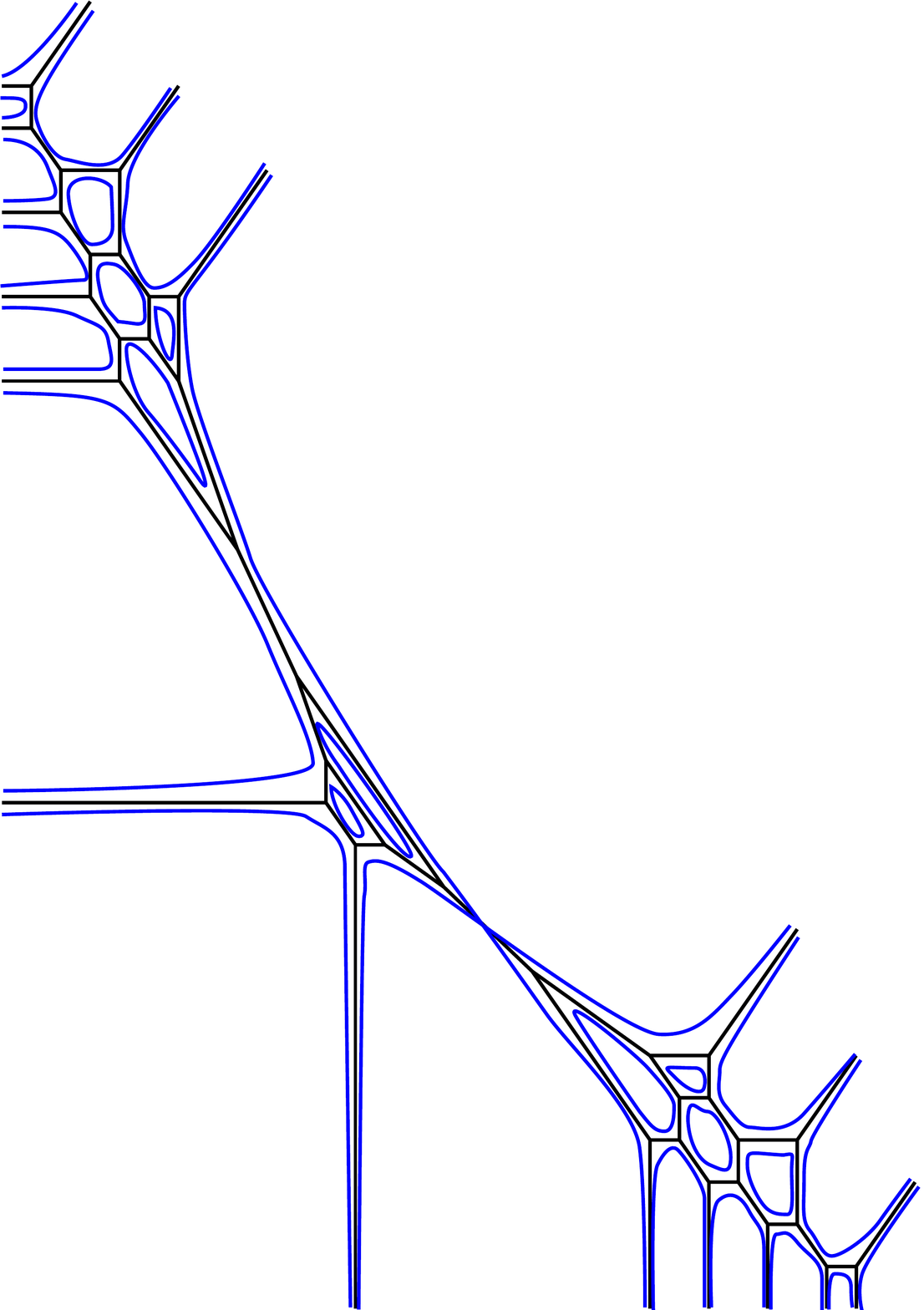}  &
\includegraphics[height=7cm, angle=0]{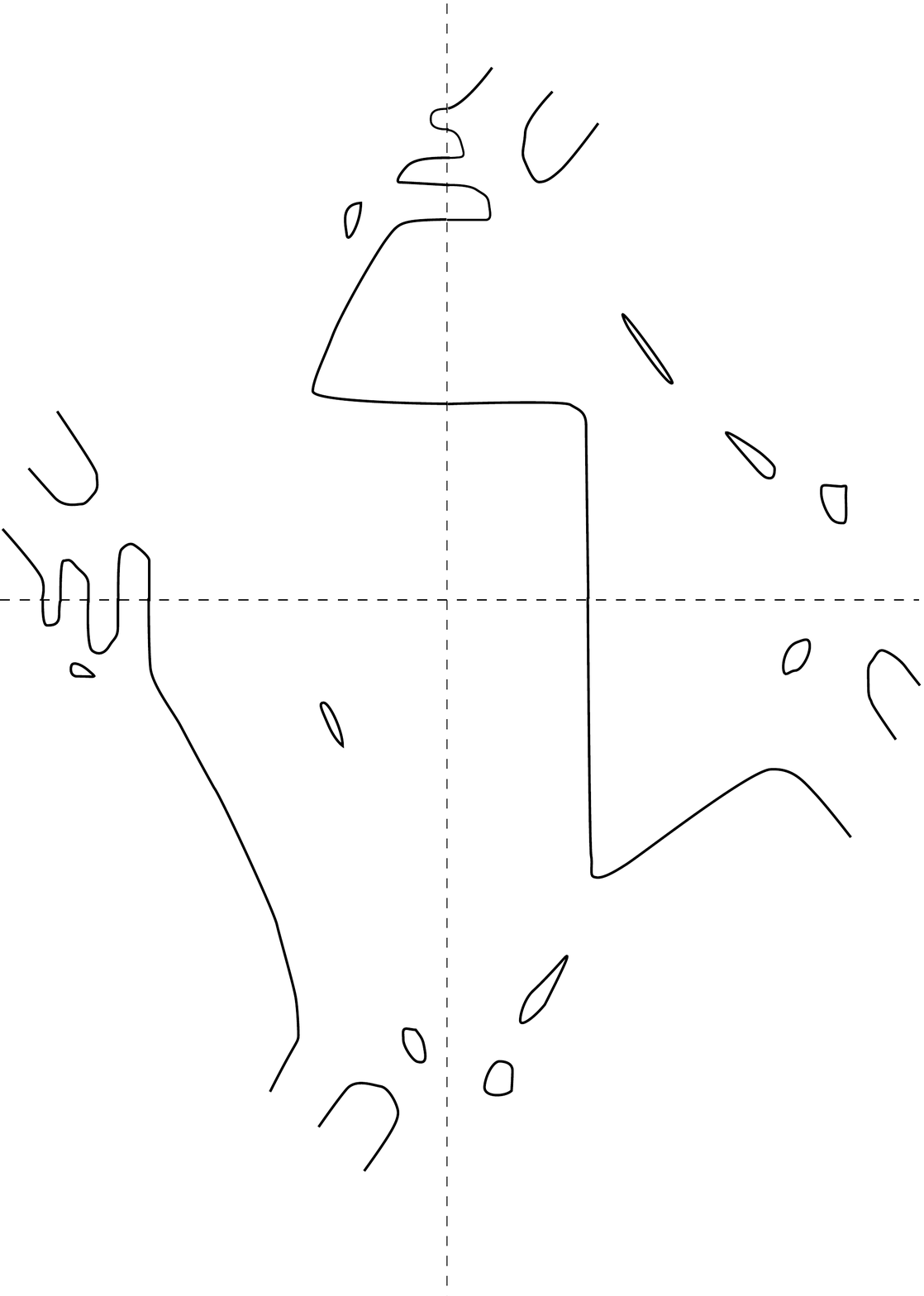}  &
\includegraphics[height=7cm, angle=0]{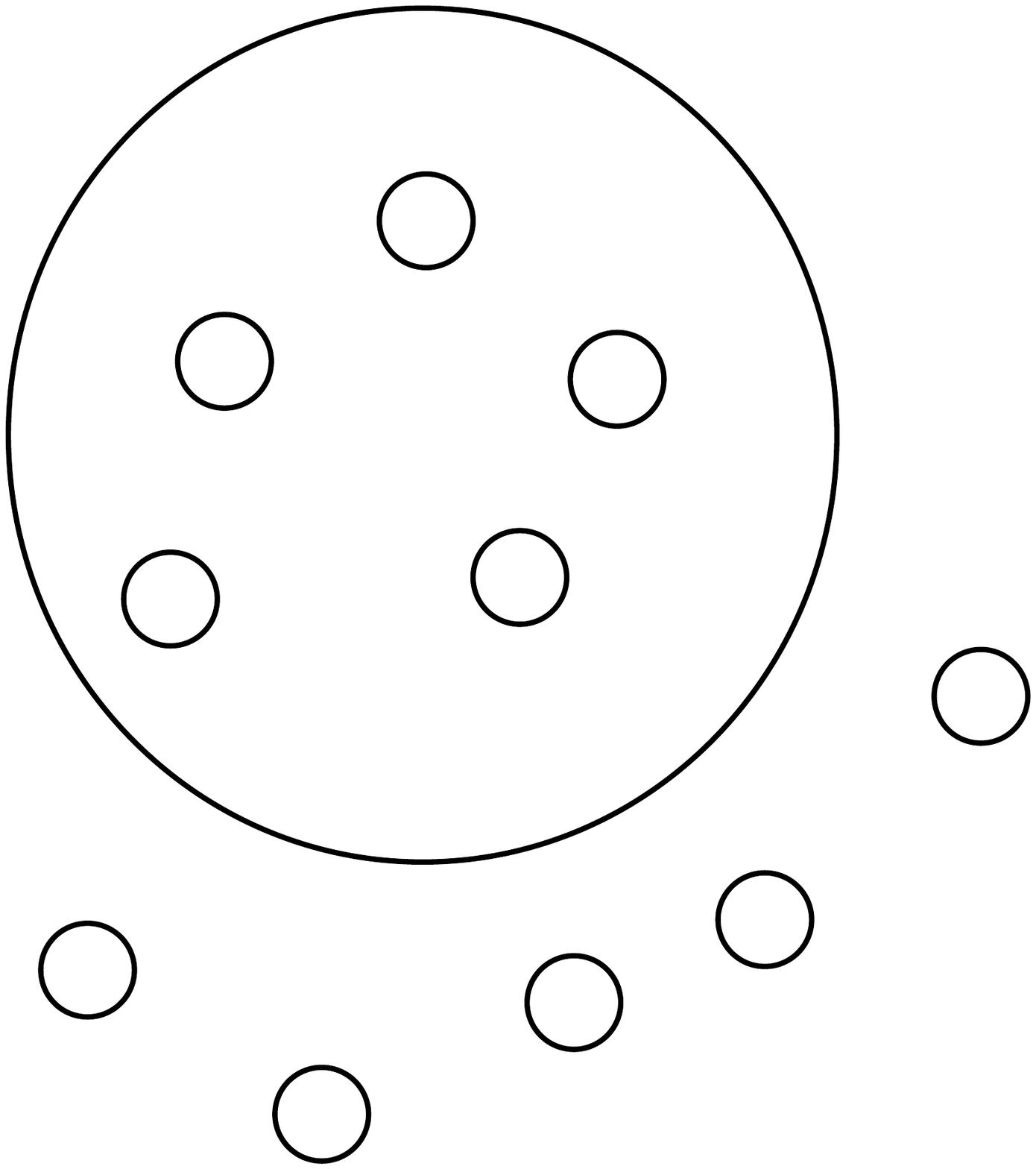}

\end{tabular}
\caption{Gudkov's sextic}\label{fig:patch gudkov}
\end{figure}

\begin{figure}[h]
\centering
\begin{tabular}{ccc}
\includegraphics[height=7cm, angle=0]{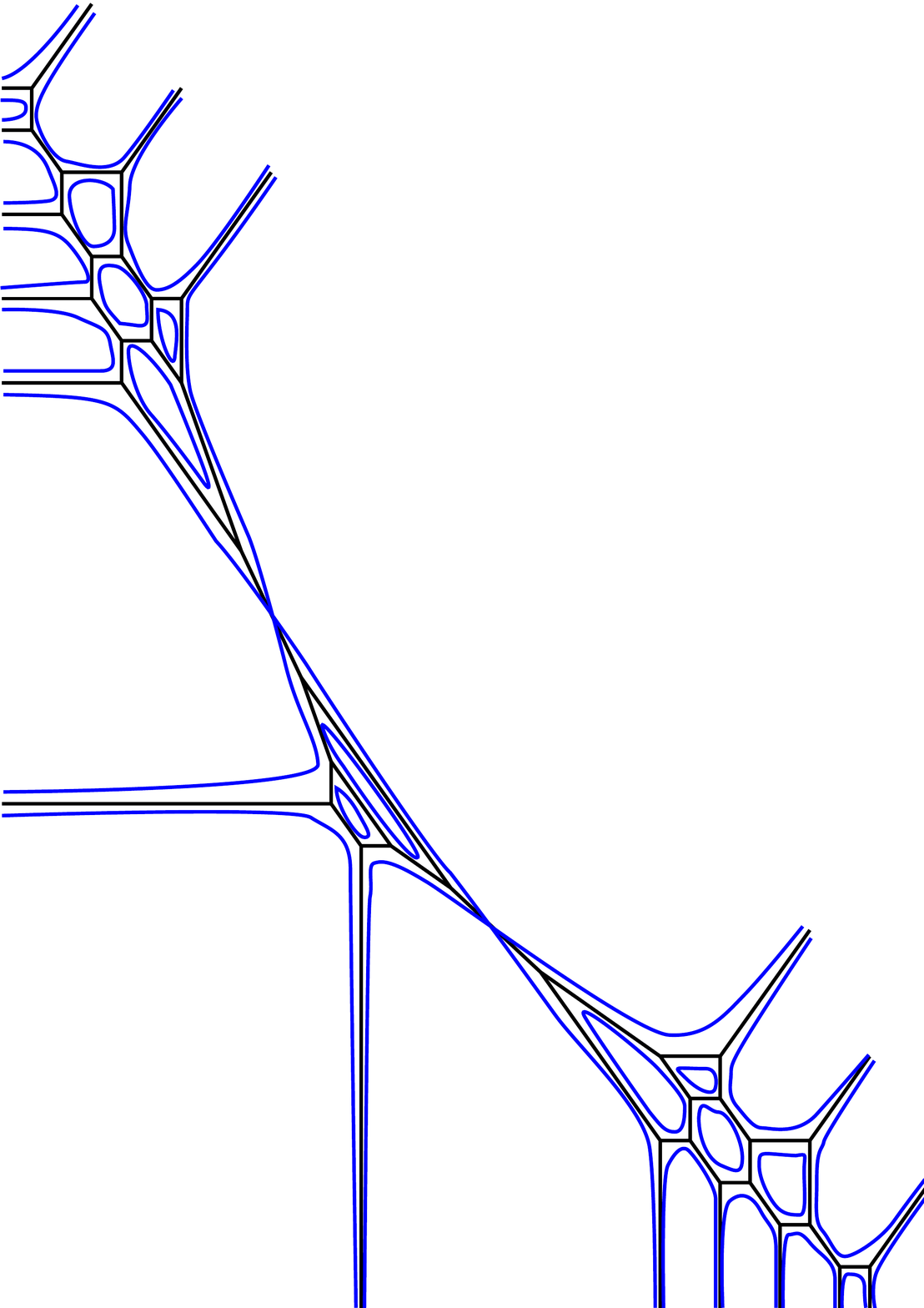}  &
\includegraphics[height=7cm, angle=0]{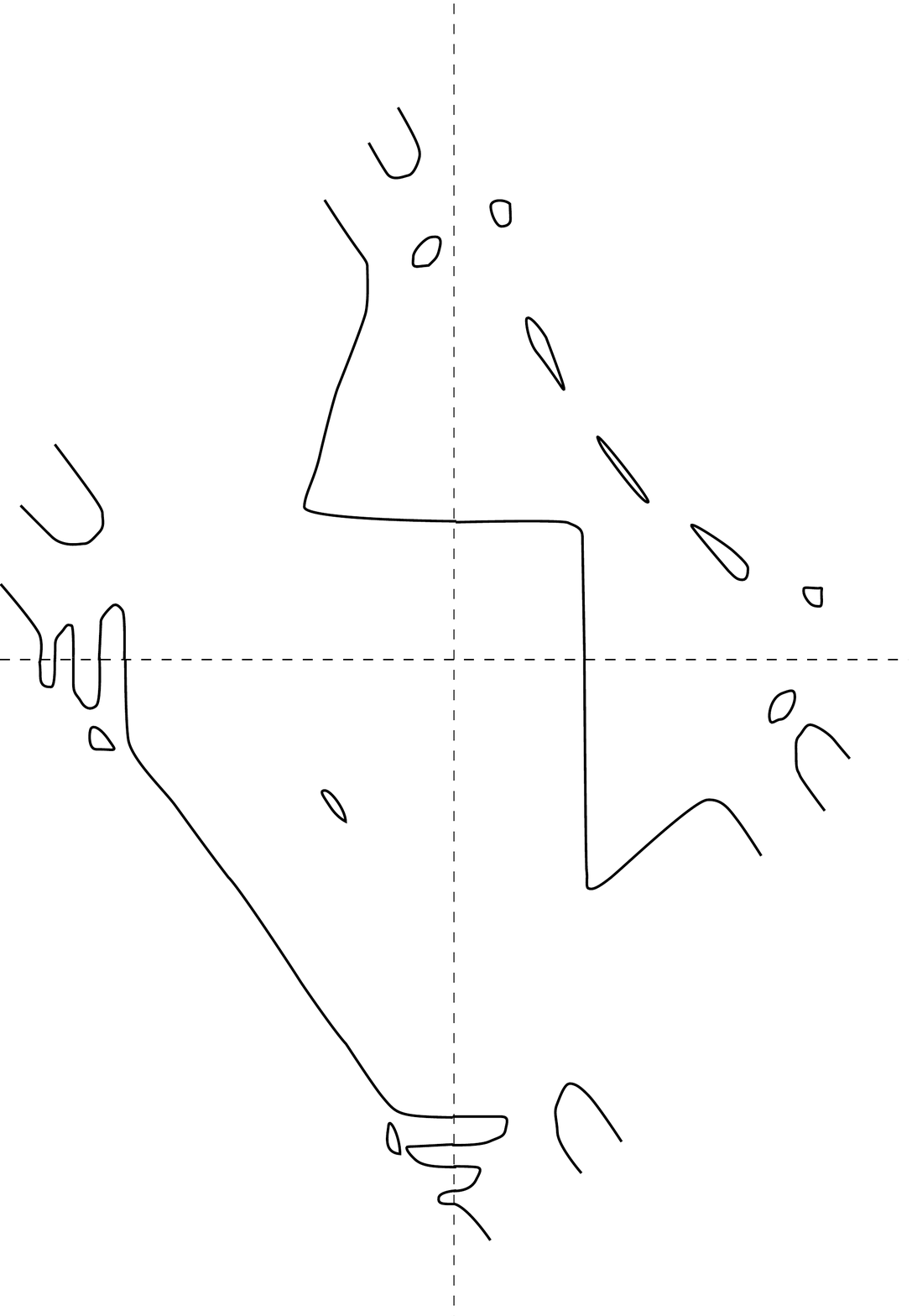} &
\includegraphics[height=7cm, angle=0]{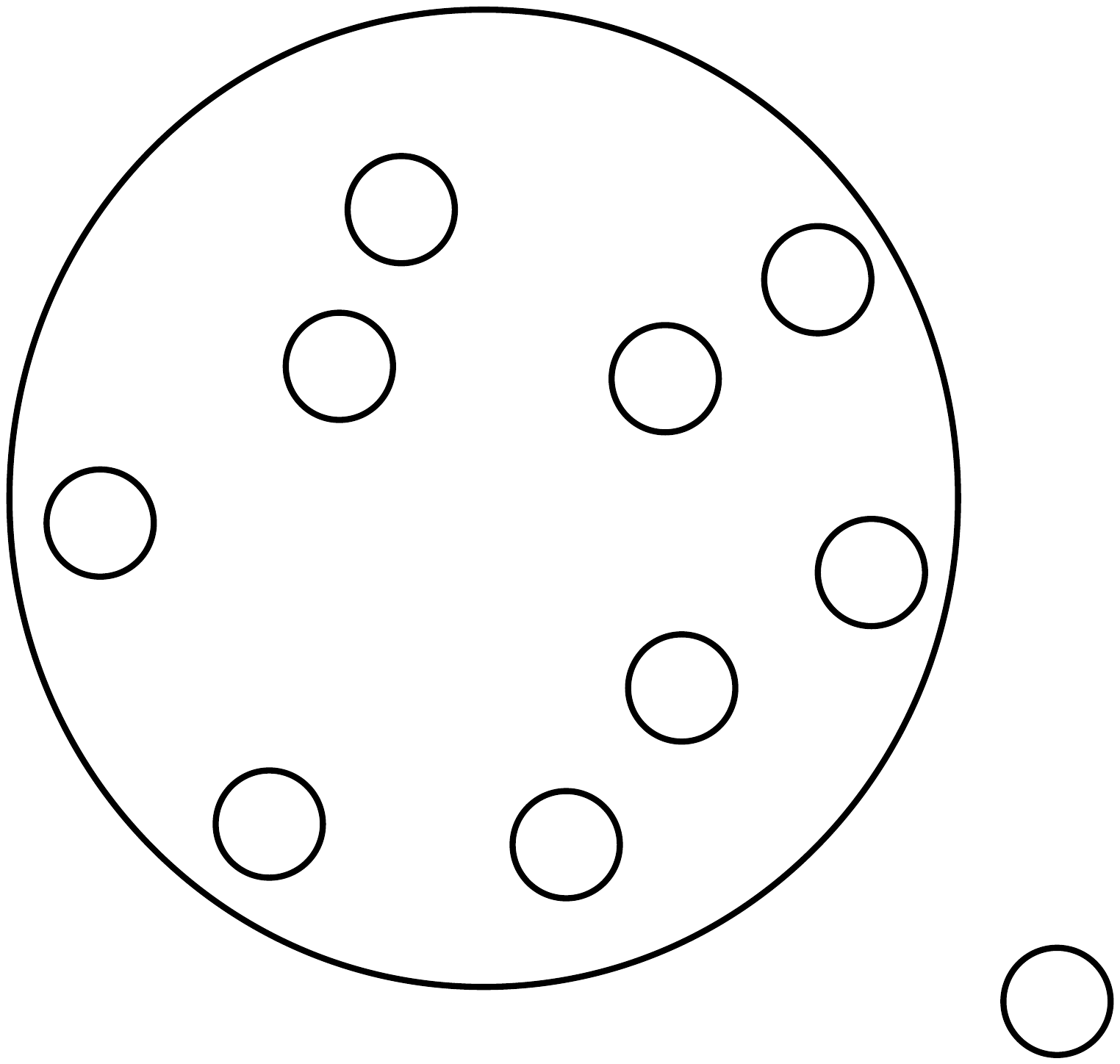}

\end{tabular}
\caption{Hilbert's sextic}\label{fig:patch hilbert}
\end{figure}
\end{example}

Given a non-singular tropical curve $C$, 
Haas' classified in \cite{Haa2} 
all twist-admissible sets
producing an $M$-curve.
At
that time the formalism of tropical geometry did not exist yet, and the original
formulation of Haas' theorem is dual to the one we present here in
Theorem \ref{thm:haas}.

We say that a twist-admissible set of edges $T$ of $C$
is \emph{maximal} if it satisfies the two following
conditions:
\begin{enumerate}
\item any cycle in $H_1(C;\Z/2\Z)$ contains an even number of edges in
  $T$;
\item  for any
edge $e\in T$, either $C\setminus e$ is disconnected, or there exists
an edge $e'\in T$ such that $C\setminus e$ and $C\setminus e'$ are
connected, but
 $C\setminus (e\cup e')$ is disconnected.
\end{enumerate}

\begin{theorem}[Haas]\label{thm:haas}
Let $C$  be a 
non-singular 
tropical curve in $\RR^2$, and let $T$ be a
twist-admissible 
set of edges of $C$.
Then 
a real algebraic curve whose existence is asserted by Theorem \ref{thm:viro}
 is maximal
 if and only if
$T$ is maximal.
\end{theorem} 

\begin{example} 
{\rm The empty collection 
is  maximal, hence we 
 recover the existence, for each non-singular tropical curve, of a canonical maximal patchworking mentioned in the introduction.}
\end{example}

\subsection{Proof of  Haas' Theorem}\label{sec:haas}

Here we deduce Theorem~\ref{thm:haas} from Theorem~\ref{thm:main} and
Proposition~\ref{prop:equiv W}. First, we  introduce some standard
notations.

The coordinatewise argument and
$\log$ 
maps are defined by:
$$\begin{array}{cccc}
\Arg: & (\C^\times)^2&\longrightarrow &(\R/2\pi\Z)^2
\\ & (z,w)&\longmapsto & \left(\text{arg}(z),\text{arg}(w)  \right)
\end{array}\quad \mbox{and}\quad
\begin{array}{cccc}
\Log: & (\C^\times)^2&\longrightarrow &\R^2
\\ & (z,w)&\longmapsto &
\left(\log|z|,\log|w|  \right)
\end{array}.$$
Note that the image of the map $\Arg$ is canonically identified with
any fiber of the map $\Log$.
We also define the following self-diffeomorphism of $(\C^\times)^2$
 fot $t>1$ 
$$\begin{array}{cccc}
H_t: & (\C^\times)^2&\longrightarrow &(\C^\times)^2
\\ & (z,w)&\longmapsto & \displaystyle \left(|z|^{\frac{1}{\log(t)}}\frac{z}{|z|}, |w|^{\frac{1}{\log(t)}}\frac{w}{|w|}\right)
\end{array}.$$
Given  a complex polynomial 
$P(z,w)=\sum a_{i,j}z^iw^j$ and $\Delta\subset\R^2$, we define
$$P^\Delta(z,w)=\sum_{(i,j)\in\Delta} a_{i,j}z^iw^j.$$
The \emph{(closed) coamoeba} of the algebraic curve $X$
with equation $P(z,w)=0$, denoted by
$\mathcal C\A(P)$, is defined
as the topological closure in $(\R/2\pi\Z)^2$ of the set 
$\Arg(X)$.

\begin{example}\label{exa:binom}
{\rm If $P(z,w)$ is a real binomial whose Newton 
 segment $\Delta(P)$ has integer length 
1, then the real part of the real algebraic curve defined by $P(z,w)$
intersects two quadrants $Q_1$ and $Q_2$ of $(\R^\times)^2$. If $u$ denotes
a primitive 
integer vector normal to  $\Delta(P)$, then
the coamoeba $\mathcal C\A(P)$ is the geodesic in $(\R/2\pi\Z)^2$ with
direction $u$ and passing through $\Arg(Q_1)$ and $\Arg(Q_2)$. On Figure~\ref{fig:coamoeba}a we depicted the coamoeba of a line given by the equation $z+aw=0$ with $a>0$. It joins the points $(0,\pi)$ and $(\pi,0)$.}
\end{example}

\begin{figure}[h!]
\begin{center}
\begin{tabular}{ccc}
\includegraphics[width=5cm, angle=0]{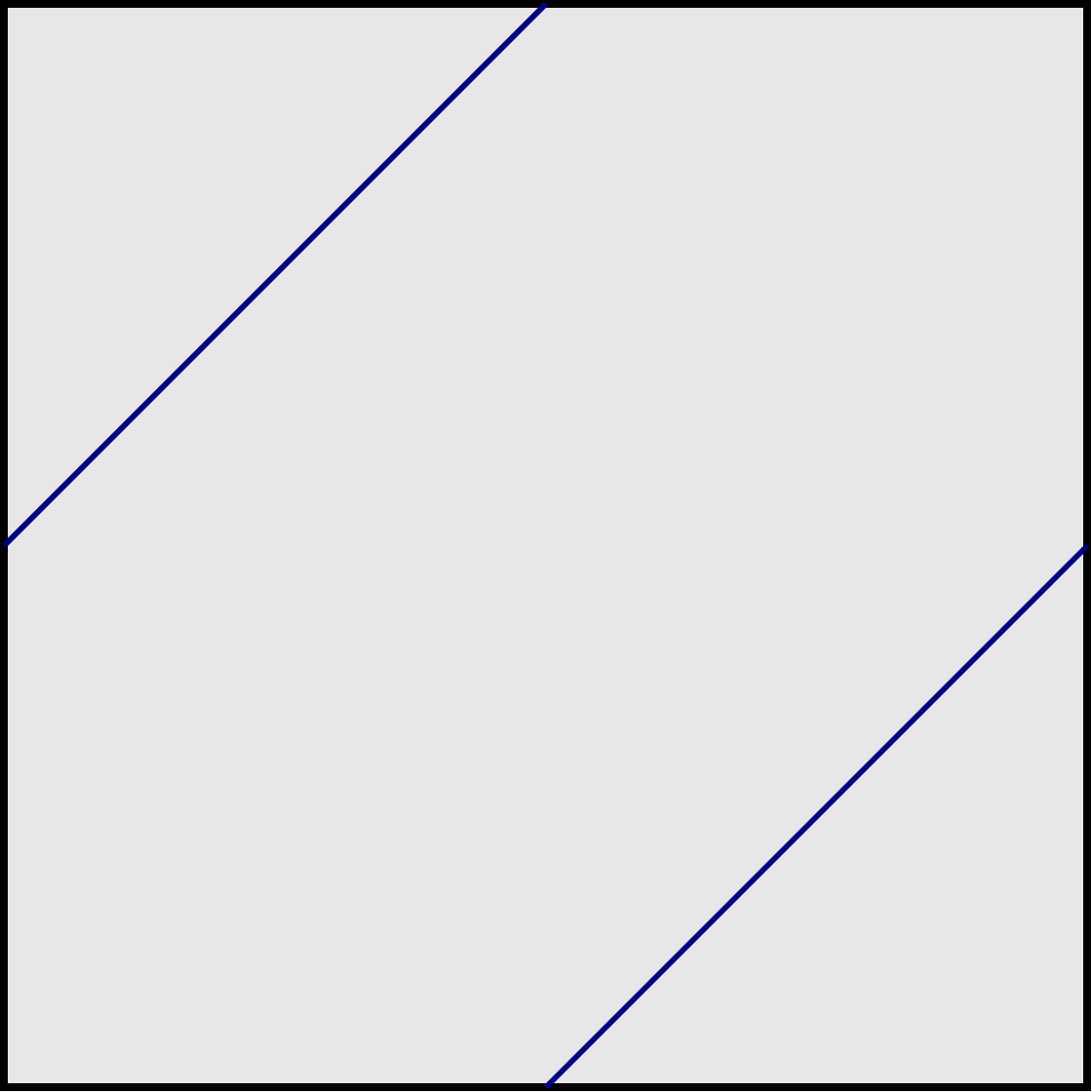}
\put(-85,-10){$(\pi,0)$}
\put(-155,-10){$(0,0)$}
\put(-170,70){$(0,\pi)$}
\put(-85,148){$(\pi,0)$}
\put(0,70){$(0,\pi)$}
& \hspace{15ex}
&\includegraphics[width=5cm, angle=0]{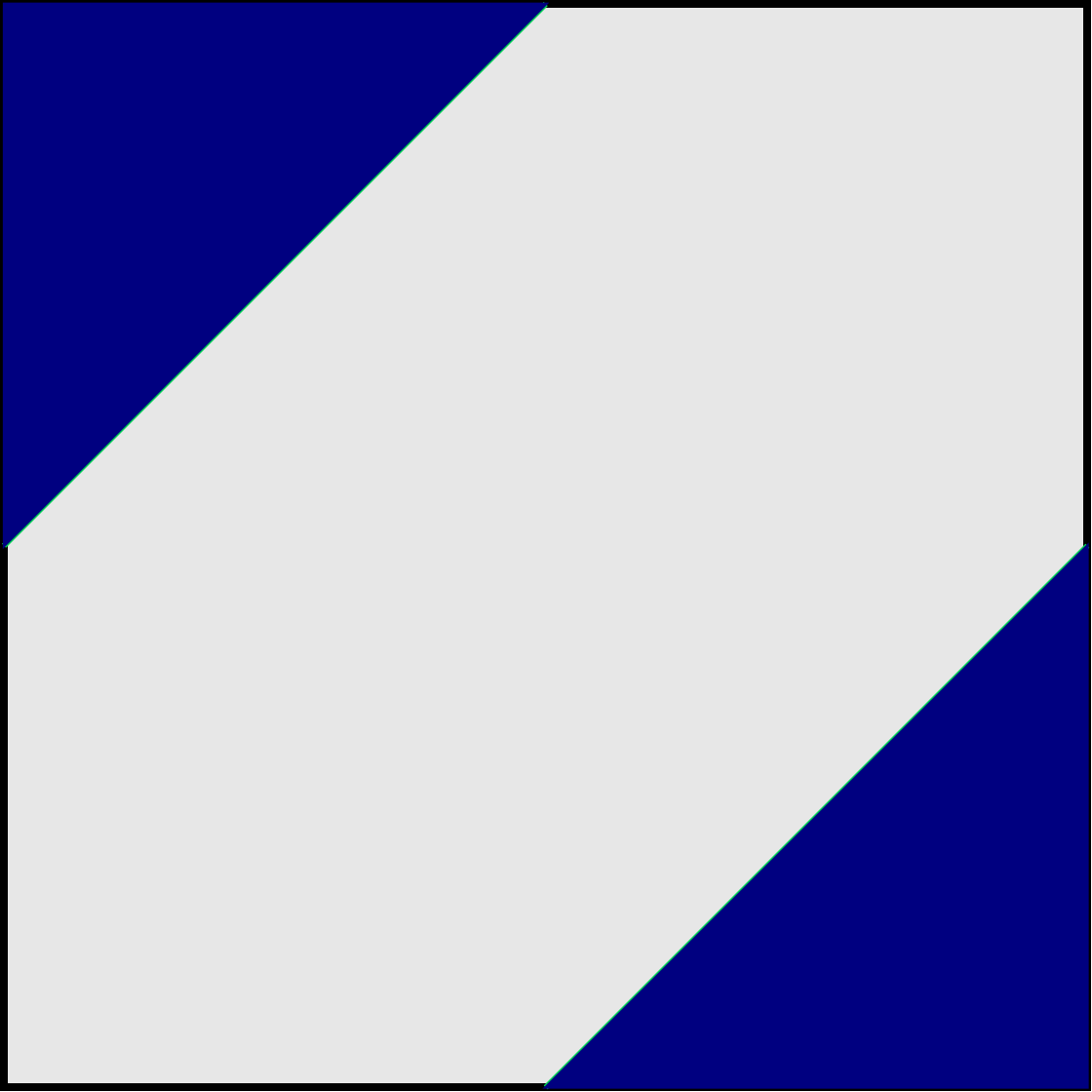}
\put(-85,-10){$(\pi,0)$}
\put(-155,-10){$(0,0)$}
\put(-170,70){$(0,\pi)$}
\put(-85,148){$(\pi,0)$}
\put(0,70){$(0,\pi)$}
\\
\\ a) Coamoeba of a real line
&&  b)  Coamoeba of a real line 
\\defined by $z+aw$ with $a>0$
&& defined by $z+aw+b$ with $a>0$ and $b<0$
\end{tabular}
\end{center}
\caption{} 
\label{fig:coamoeba}
\end{figure}

\begin{example}[See {\cite[Proposition 6.11 and Lemma 8.19]{Mik1}}]\label{exa:trinom}
{\rm  If $P(z,w)$ is a real trinomial whose Newton triangle $\Delta(P)$  has Euclidean area 
$\frac{1}{2}$, then  the real part of the real algebraic curve defined
by $P(z,w)$ 
intersects three quadrants $Q_1,Q_2$, and $Q_3$ of $(\R^\times)^2$. 
If $u_1, u_2,$ and $u_3$ denote
 primitive 
integer vectors normal to the three edges of  $\Delta(P)$, 
then
the coamoeba $\mathcal C\A(P)$ is the union of the two triangles
with vertices  $\Arg(Q_1)$, $\Arg(Q_2)$,  and
 $\Arg(Q_3)$, and
whose sides are geodesics with direction $u_1$, $u_2$, and $u_3$.
In particular  $\mathcal C\A(P)$ is a (degenerate) pair of pants.
On Figure~\ref{fig:coamoeba}b we represented the coamoeba of a line given by the equation  $z+aw+b=0$ with $a>0$ and $b<0$.}
\end{example}

Recall (see for example \cite[Section 2.2]{BIMS15}) that to
any tropical curve $C$ in $\R^2$ is associated a dual subdivision of
its Newton polygon. We denote respectively by $\Delta_v$ and
$\Delta_e$ the polygon dual to 
the vertex $v$ and the
edge $e$ of $C$. 
Next statement is proved in  \cite[Section 6]{Mik1}.
\begin{theorem}[Mikhalkin]\label{thm:patch mikh}
Let $C$ be a
tropical curve in $\R^2$ defined by some
tropical polynomial 
$\tg \sum b_{i,j}x^iy^j \td $.
Then given a collection $a_{i,j}$ of
non-zero complex numbers, the image by $H_t$ of the algebraic curve in 
$(\C^\times)^2$ defined by the complex polynomial 
$$P_t(z,w)=\sum a_{i,j}t^{-b_{i,j}}z^iw^j $$
converges when $t\to+\infty$, for the Hausdorff metric on compact subsets of $(\C^\times)^2$, to a subset
$V_\infty$ that can be described as follows:
\begin{itemize}
\item $\Log(V_\infty)=C$;

\item for any vertex $v$ of $C$,
  we have $\Log^{-1}(v)\cap V_\infty=\mathcal C\A(P_1^{\Delta_v})$;

\item for each edge $e$ of $C$,

  we have $\Log^{-1}(e)\cap V_\infty=e\times \mathcal C\A(P_1^{\Delta_e})$ .
\end{itemize}
If furthermore all the $a_{i,j}$'s are real numbers, then the real part
of the real algebraic curve defined by $P_t(z,w)$ converges when
$t\to+\infty$, for the Hausdorff metric on
compact
subsets of $(\R^\times)^2$, to $V_\infty\cap (\R^\times)^2$.
\end{theorem}

Note that one can recover Theorem \ref{thm:viro} by combining
 Theorem
\ref{thm:patch mikh} together with Examples  \ref{exa:binom} and \ref{exa:trinom}.
In particular, an equation of a simple Harnack curve is given by next
proposition.
We define the function $\epsilon:\Z^2\to \{\pm1\}$ by $\epsilon(i,j)=1$
  if both $i$ and $j$ are even, and by $\epsilon(i,j)=-1$ otherwise.
Recall that the Viro's patchworking construction and the definition of
a twist admissible
set of edges of a non-singular tropical curve are given in Section \ref{sec:viro}.
\begin{proposition}[Itenberg, see \cite{IV2} or {\cite[Remark
        3.9]{BIMS15}}]\label{prop: harnack signs}
Let $C$ be a non-singular 
tropical curve in $\R^2$ defined by some
tropical polynomial 
$\tg \sum b_{i,j}x^iy^j \td $, and let 
$$P_t(z,w)=\sum  \epsilon(i,j)t^{-b_{i,j}}z^iw^j, \quad
V_\infty=\lim_{t\to+\infty}H_t\left(\{P_t=0\}\right).
$$
Then  up to axial symmetries, the real part of $V_\infty$ 
is isotopic in $(\R^\times)^2$  to the curve $C_\emptyset$ constructed out
of $C$ and the empty twist admissible set. 
In particular, the real algebraic curve defined by $P_t(z,w)$ with $t$
large enough is maximal. 
 Furthermore, the
isotopy type of its real part  in $(\R^\times)^2$, up to axial symmetries, only depends on the Newton
polygon of $C$.
\end{proposition}

We are now ready to deduce Haas' Theorem from what is discussed above.
\begin{proof}[Proof of Theorem \ref{thm:haas}]
Let $C$ be a non-singular 
tropical curve in $\R^2$ defined by the
tropical polynomial 
$\tg \sum b_{i,j}x^iy^j \td $, and let  $\nu:\Z^2\to \{\pm 1\}$ be
some function. 
We define
$$P_t(z,w)=\sum  \epsilon(i,j)t^{-b_{i,j}}z^iw^j, \quad
V_\infty^0=\lim_{t\to+\infty}H_t\left(\{P_t=0\}\right) ,
$$
and
$$R_t(z,w)=\sum  \nu(i,j)t^{-b_{i,j}}z^iw^j,\quad
 V_\infty=\lim_{t\to+\infty}H_t\left(\{R_t=0\}\right). $$
We equipped both $V_\infty$ and  $V_\infty^0$ with the real structure coming from the
restriction of the complex conjugation on $(\C^\times)^2$.
Compactifying $C$ by gluing a point to each unbounded edge $C$, we
obtain
a graph $\Gamma$.
The map $\Log:V_\infty^0\to C$ induces a pair of pants decomposition of
the topological surface $V_\infty^0$. To each unbounded edge of $C$
corresponds an unbounded cylinder of  $V_\infty^0$. 
A surface $S_\Gamma$ as in Section \ref{sec:graph surf} is obtained by
gluing a disk to each such unbounded cylinder of  $V_\infty^0$, and we
have the identities
$$S_v= \Log^{-1} (v)\cap V^0_\infty\ \ \forall v\in \Ve(C), \quad
\mbox{and}\quad S_e=\Log^{-1} (e)\cap V^0_\infty  \ \ \forall e\in\Ed(C),$$
 up to considering degenerate pairs of pants for $S_v$ instead of usual ones in the construction of $S_\Gamma$ in Section \ref{sec:graph surf}. All previous definitions and results are easily seen to hold with this benign substitution.
 Furthermore, the real structure on   $V_\infty^0$ induces a real structure
$\tau_{\Gamma,1}$ 
 above 
the real graph $(\Gamma,\Id)$.

Since $C$ in non-singular, for any vertex $v$  of $C$ there exists 
 $(\mu_1,\mu_2)\in\{\pm 1\}^2$ such that
$$R_t^{\Delta_v}(\mu_1 z, \mu_2 w)= \pm P_t^{\Delta_v}( z, w). $$
In particular, the  map $s_v:(z,w)\mapsto (\mu_1z,\mu_2w)$ induces a
real homeomorphism (i.e. commuting with  real structures)
$$\zeta_v:\Log^{-1} (v)\cap V_\infty \to S_v.$$
Let $\zeta:V_\infty\to V_\infty^0$ be a (not necessarily real)
homeomorphism restricting to $\zeta_v$ on $\Log^{-1} (v)\cap V_\infty$
for each vertex $v$ of $C$. Note that $\zeta$ is well defined up to
isotopy and
  conjugation by a finite product of  Dehn twists along cylinders
  $S_e$ with $e\in\Ed^0(\Gamma)$.
Furthermore, 
given an edge $e\in \Ed^0(\Gamma)$ adjacent to two vertices $v$ and $v'$,
the two maps $s_v$ and $s_{v'}$  coincide if and only if, up to 
an isotopy restricting on the identity on $\partial S_e$,
 the restriction of
 $\zeta$ to $\Log^{-1} (e)\cap V_\infty$ is a real map. 
 The real structure on $V_\infty$ induces, via the map $\zeta$,
a real structure on $V_\infty^0$, which in its
turn induces a real structure $\tau_{\Gamma,2}$ 
 above the real graph $(\Gamma,\Id)$.
Let $T$ be the set of edges for which  $\zeta\restrict{\Log^{-1} (e)\cap
V_\infty}$ is not a real map. 
By Lemma \ref{lem:2 patch}, we have 
$$\tau_{\Gamma,2}-\tau_{\Gamma,1}=\sum_{e\in T}e.$$
On the other hand, by {\cite[Theorem 3.4 and Remark 3.9]{BIMS15}},
the set $T$ is twist-admissible and the real part of the  algebraic curve in 
$(\C^\times)^2$ with equation 
$R_t(z,w) =0$
is isotopic in $(\R^\times)^2$  to the curve $C_T$ for $t$ large enough.
Hence by Corollary \ref{cor:patch max} and Proposition \ref{prop:equiv
  W} applied to $\tau_{\Gamma,1}$ and $\tau_{\Gamma,2}$, 
the real  algebraic curve in 
$R_t(z,w) =0$ for $t$ large enough is maximal if and only if $T$ is maximal.

\medskip
To  finish 
 the proof of Haas' Theorem, it is enough to notice that for any set $T$ of
twist-admissible edges  of $C$, there exists a function   $\nu:\Z^2\to
\{\pm 1\}$ as above such that the real part of the  algebraic curve in 
$(\C^\times)^2$ with equation 
$$\sum  \nu(i,j)t^{-b_{i,j}}z^iw^j =0$$
is isotopic in $(\R^\times)^2$  to the curve $C_T$ for $t$ large enough,
see
{\cite[Remark 3.9]{BIMS15}}. \hfill $\Box$
\end{proof}

\providecommand{\bysame}{\leavevmode\hbox to3em{\hrulefill}\thinspace}
%
%

\bibliographystyle{amsalpha}
\bibliographymark{References}

\end{document}